%% file: main.tex
\documentclass[12pt]{article} 
\usepackage{fullpage}

\usepackage{algorithm}
\usepackage{algorithmic}

\usepackage[align=center,shadow=true,shadowsize=5pt,nobreak=true,framemethod=tikz,style=0,skipabove=2pt,skipbelow=1pt,innertopmargin=-3pt,innerbottommargin=7pt,innerleftmargin=6pt,innerrightmargin=6pt,leftmargin=-2pt,rightmargin=-2pt]{mdframed}
\usepackage{natbib}
 \input{00notation}

\date{} 

\begin{document}

\title{Riemannian Perspective on Matrix Factorization}

\author[1]{Kwangjun Ahn%
\thanks{
 Email: {\tt kjahn@mit.edu}. This work was supported by graduate assistantship from the NSF Grant (CAREER: 1846088) and by Kwanjeong Educational Foundation.}%
}
\author[2]{Felipe Suarez%
\thanks{Email: {\tt felipesc@mit.edu}. This work was supported in part by a graduate assistantship from the NSF award IIS-1838071.}%
}
\affil[1]{Department of  Electrical Engineering and Computer Science, MIT}
\affil[2]{Department of Mathematics, MIT}
\maketitle
\begin{abstract}
\input{11abstract}
\end{abstract}

\section{Introduction}
 
Matrix completion is a classical  problem in machine learning and signal processing  that aims to recover an unknown low-rank matrix from only a few observed entries.
Ever since the pioneering work by~\cite{candes2009exact}, there have been a flurry of works solving matrix completion with guarantees. 
See  a survey by~\cite{candes2012survey} and the introduction of~\citep{ge2016matrix} for detailed information.

Among many approaches, one prominent approach widely used in practice is based on  matrix factorizations, \emph{\`{a} la}~\cite{burer2003nonlinear}.
Letting $M$ be the  $m\times n$ unknown matrix  of rank $r$, the  matrix factorization approach solves
\begin{align} \label{matrix_factorization}
    \mini_{\substack{\VX\in\R^{m\times r}\\ 
    \VY\in \R^{n\times r}}} \frac{1}{2}\sum_{(i,j): \text{ observed}}\left[(\VX\VY^\top)_{i,j} -\VM_{i,j}\right]^2.
\end{align}
Due to  explicit factorization, the matrix factorization approach is computationally advantageous, making it widely applicable in practice~\citep{koren2009bellkor}.

To complement its success in practice, there have been a great number of works trying to understand the matrix factorization approach \emph{theoretically}~\citep{keshavan2010matrix,jain2013low,hardt2014understanding,hardt2014fast,sun2015guaranteed,zhao2015nonconvex,chen2015fast,de2015global,ge2016matrix, bhojanapalli2016global,ge2017no,park2017non}.
As discussed in the seminal work by  \cite{sun2015guaranteed}, the main difficulty lies in the non-convex nature of the problem, arising from the problem symmetry.
For instance,  if a pair $(\VX^\star,\VY^\star)$ achieves the global optimum of \eqref{matrix_factorization}, then it follows from the symmetry that  $(\VX^\star \VR,\VY^\star(\VR^{-1})^\top)$ for any invertible $\VR$ also achieves the optimum.
In particular, the non-convexity of matrix factorization eludes conventional theoretical frameworks based on convex analysis.

In this work, we study the matrix factorization approach from a  Riemannian geometric perspective.
Inspired by previous works~\citep{dai2011subspace,boumal2011rtrmc,dai2012geometric, boumal2015low,pitaval2015convergence}, we consider a formulation of the matrix factorization approach as an optimization over a single Grassmannian.
We then investigate the Riemannian geometry of the cost function.
For the fully observed case, our main result offers a crisp characterization of the landscape  based on the  notion of   principal angles between subspaces.
More specifically, our result characterizes a region in which the cost function is geodesically convex and outside of  which the cost function has an escaping direction, i.e. a direction along which the second derivative is  negative. 
In particular, all  critical points outside this region are strictly saddle.
Lastly, we empirically study  the partially observed case and observe that the landscape resembles the fully observed case for the settings considered in previous works.

\subsection{Related work}

Over the last decade, there have been a myriad of works solving matrix completion via \emph{Riemannian optimization}\footnote{See a recent monograph \citep{boumal2020intromanifolds} for a gentle introduction on  optimization over Riemannian manifolds.}~\citep{edelman1998geometry, absil2004riemannian}.
Previous methodological contributions have employed various optimization methods, such as (stochastic) gradient descent~\citep{keshavan2009gradient,keshavan2010matrix,balzano2010online,mishra2012riemannian,mishra2014fixed}, conjugate gradient methods~\citep{vandereycken2013low,boumal2011rtrmc,mishra2012riemannian,boumal2015low, cambier2016robust}, trust-region methods~\citep{boumal2011rtrmc,mishra2012riemannian,mishra2014fixed,boumal2015low}, Newton's method~\citep{simonsson2010grassmann}, scaled gradient methods~\citep{ngo2012scaled}, and others.
Only a few works have studied the theoretical aspect of the Riemannian optimization approach. The closest result to ours can be found in~\citep{pitaval2015convergence} where they show the global convergence of a Riemannian gradient flow.
Another related result can be found in~\citep[\S 6]{keshavan2010matrix} where they study \emph{local} geometry of their formulation based on \emph{two} Grassmannian manifolds. 
Other works study formulations based on \emph{different} manifolds. 
For instance, \citep{wei2016guarantees} and \citep{hou2020fast} study the formulation based on the manifold of fixed rank matrices. 

Apart from matrix completion, over the last few years, there have been many works understanding Riemannian landscape  for other problems, including dictionary learning~\citep{sun2016complete,bai2018subgradient,zhu2019linearly,li2019weakly}, (robust) subspace recovery~\citep{maunu2019well,zhu2019linearly,li2019weakly}, matrix sensing~\citep{hou2020fast}, and point-set registration~\citep{bohorquez2020maximizing}.

\subsection{Setting and notations}

\paragraph{Setting.} We use bold lower-case letters (e.g. $\vu,\vv,\vx,\vy$) to denote vectors, and bold upper-case letters (e.g. $\VU,\VV,\VX,\VY$) to denote matrices, and reserve $\VZE$ for the zero  matrix. 
We assume that the ground truth matrix $\VM$ is a  $m\times n$ matrix ($m\leq n$) of rank $r$ whose  singular value decomposition is $\VU \VS \VV^\top$.
Let $\smin$  and $\smax$ be the smallest and largest  singular value, respectively. 
Let $\pp \subset [m]\times [n]$ be the subset of observed positions.
 For $\VA\in \R^{m\times n}$, $\VA_\pp$ is the $m\times n$ matrix defined as $(\VA_\pp)_{i,j}=\VA_{i,j}$ if $(i,j)\in \pp$ and 
 $(\VA_\pp)_{i,j}=0$ otherwise.
$\inpp{\VA}{\VB}$ denotes the inner product $\Tr[(\VA_\pp)^\top \VB_\pp]$ and 
$\pnorm{\VA}^2 := \inpp{\VA}{\VA}$.

\paragraph{Other notations.} 
The column space is denoted  $\col(\cdot)$.
The $m\times m$ identity matrix is denoted $\VI_{m}$.
For a function $h:\R\to \R$ and a diagonal matrix $\VD=\diag(d_1,d_2,\dots,d_r)$, $h(\VD):=\diag(h(d_1),h(d_2),\dots, h(d_r))$.

\section{Background on Grassmannian manifolds} \label{sec:background}

Throughout this paper, we consider a  Riemannian manifold on the space of subspaces called the Grassmannian manifold~\citep{grassmann1862ausdehnungslehre}.  
Here we provide a brief background. For more details, we refer the readers to the seminal works by~\cite{edelman1998geometry} and~\cite{absil2004riemannian}; see also \citep[Ch. 9]{boumal2020intromanifolds}.
 
\subsection{Riemannian geometry of Grassmannian manifolds}
\paragraph{Grassmannian manifold.}  
For $r \leq m$, the Grassmannian manifold  $\Gr(m,r)$ is the set of $r$-dimensional subspaces in $\R^m$.
We consider the canonical Euclidean embedding into $\R^{m\times r}$, in which each point is represented by the equivalence class
 \begin{align} \label{def:grassmann}
     [\VX] = \left\{\VX \VQ~:~ \VQ \in \orth(r)\right\} 
 \end{align}
 for $\VX\in\R^{m\times r}$ such that $\VX^\top \VX  = \VI_r$.
 An equivalent way of representing $\Gr(m,r)$ is via the   quotient $\orth(m)/(\orth(r) \times \orth(m-r))$, where $\orth(d)$ is the orthogonal group of $d\times d$ matrices~\citep{edelman1998geometry}. 
 \begin{remark}
From now on, we will often abuse notation and identify $\VX$ with the entire equivalence class $[\VX]$ whenever it is clear from the context.   
 \end{remark}

We now describe the Riemannian geometry of the Grassmann manifold.
Since the Grassmannian manifold is a quotient manifold, for concreteness, it is convenient to work with a specific representative of the equivalent class.
In the following,  we choose and fix a representative element $\VX$ for a point $[\VX]\in \Gr(m,r)$.

\paragraph{Tangent space.}  First, the tangent space  $\T{\VX}\Gr(m,r)$ at $\VX$ is given as 
\begin{align*}
    \T{\VX}\Gr(m,r):= \left\{\dc \in \R^{m\times r} ~:~ \VX^\top \dc =\VZE \right\}\,.
\end{align*}
For any direction $\VD\in \R^{m\times r}$, the projection of it onto the tangent space at $\VX$ is given as 
\begin{align*} 
(\VI_m -\VX\VX^\top) \VD \in \T{\VX}\Gr(m,r)\,.
\end{align*}

\paragraph{Metric.}The Riemannian metric on the Grassmann manifold is given as 
\begin{align*}
    \inp{\dc_1}{\dc_2} := \Tr(\dc_1^\top \dc_2 ) 
\end{align*}
for $\dc_1,\dc_2\in \T{\VX}\Gr(m,r)$.

\paragraph{Calculus.}Now, we discuss  Riemannian calculus on Grassmann manifolds.
For a function $f:\Gr(m,r)\to \R$, we define its gradient $\grad f$ to be the vector field such that
\begin{align*}
    \inp{\grad f(\VX)}{\dc} = D f(\VX)[\dc]
\end{align*}
 for any $\dc\in \T{\VX}\Gr(m,r)$,
where $ D f(\VX)[\dc]$ denotes the (Euclidean) directional derivative of $f$ along the direction $\dc$.
Moreover, the Hessian $\hess f$ is the quadratic form (or the $(0,2)$-tensor) such that for any $\dc\in \T{\VX}\Gr(m,r)$,
\begin{align*}
     \hess f(\VX)[\dc,\dc]
     = D(D f(\VX)[\dc])[\dc]. 
\end{align*}

\subsection{Principal angles, distances and geodesics} \label{sec:principal}

\paragraph{Principal angles.} In order to better understand Grassmannian geometry, we discuss  a well-established notion from linear algebra called the  \emph{principal angles} between subspaces~\citep{jordan1875essai}. 
From \eqref{def:grassmann}, one can interpret each point $\VX$ on $\Gr(m,r)$ as an equivalence class of orthonormal bases of $\col(\VX)$. 
In other words, one can canonically identify each point on $\Gr(m,r)$ with a $r$-dimensional subspace of $\R^m$.
Now having this identification, we make the following definition about the principal angles.
\begin{definition}[principal angles]
 For $[\VX],[\VY]$ on $\Gr(m,r)$, the principal angles between $[\VX]$ and $[\VY]$ are defined as the principal angles between the subspaces $\col(\VX)$ and $\col(\VY)$. 
 In other words, denoting the singular values of  $\VX^\top \VY$ by $\lambda_1\geq \dots \geq \lambda_r$, the principal angle matrix $\VTH\in \R^{r\times r}$ is the diagonal matrix whose $(i,i)$-th entry is $\arccos(\lambda_i)\in[0,\pi/2]$. We call each diagonal entry a principal angle.
\end{definition}

\noindent The following proposition is an immediate consequence of the definition.
\begin{proposition}[principal alignment]\label{prop:diagonal}
For any given two points  $[\VX],[\VY]$ on $\Gr(m,r)$, one can find two representatives $\VX\st\in [\VX]$ and $\VY\st \in [\VY]$ such that $\VX\st^\top \VY\st = \cos(\VTH)$ where $\VTH$ is the principal angle matrix between $[\VX]$ and $[\VY]$.
\end{proposition}
\begin{proof}
Choose arbitrary representatives, say $\VX$ and $\VY$, of $[\VX]$ and $[\VY]$, respectively. Let the singular value decomposition of $\VX^\top\VY$ be $\VP \VLA \VQ^\top$. 
Now choose $\VX\st := \VX\VP$ and $\VY\st := \VY\VQ$.
Then, we have $\VX\st^\top \VY\st = \VP^\top \VX^\top \VY \VQ = \VP^\top (\VP \VLA \VQ^\top) \VQ = \VLA$, as desired.
\end{proof}

 \paragraph{Distances between subspaces.} With the principal angles between subspaces, there are several well-established notions of distances; see \citep[\S 4.3]{edelman1998geometry}. 
Here we write them as distances on $\Gr(m,r)$.
\begin{definition}\label{def:distances}
  For $[\VX],[\VY]\in\Gr(m,r)$, let $\VTH$ be the principal angle matrix between $[\VX]$ and $[\VY]$. We define the following distances: 
  \begin{align}
      \darc([\VX],[\VY]) &:= \fnorm{\VTH} &&(\text{arc-length distance}),\\
      \dchor([\VX],[\VY]) &:= \fnorm{2\sin( \VTH/2)} &&(\text{chordal distance}),\\
       \dproj([\VX],[\VY]) &:= \fnorm{\sin( \VTH)} &&(\text{projection distance}).
  \end{align}
\end{definition}
\noindent In fact, the arc-length distance is equal to the Riemannian distance (the distance determined by the Riemannian metric). 
Moreover, one can easily verify that the other two distances  are equivalent metrics; see, e.g., \citep[Remark 6.1]{keshavan2010matrix}.

 \paragraph{Geodesics.}  
 It is known that when $m,r\geq 2$ there are at least countably many geodesics between  two points on $\Gr(m,r)$~\citep{wong1967differential}.
 On the other hand, if all the principal angles between the two points are less than $\pi/2$, there is a unique geodesic of the shortest length~\citep{wong1967differential}. 

Given two points $[\VX],[\VY] \in \Gr(m,r)$, let $\VTH$ be the principal angle matrix between $[\VX]$ and $[\VY]$.
Assume that all principal angles are less than $\pi/2$.
Choosing representatives $\VX\st \in [\VX]$ and $\VY\st \in [\VY]$ as per Proposition~\ref{prop:diagonal}, let $\dc\st \in \R^{m\times r}$ be the orthonormal matrix such that
\begin{align} \label{exp:geodesic_setup}
    \VY\st = \VX\st \cos(\VTH) + \dc\st \sin(\VTH)\,. 
\end{align}
Such an orthonormal matrix $\dc\st $ exists because $(\VI_m-\VX\st\VX\st^\top)\VY\st$ belongs to the tangent space $\T{\VX\st}\Gr(m,r)$.
Then, with such choices of representatives, the geodesic $\VG\st(t)$ of the shortest length joining the two points is given as
  \begin{align} \label{def:geodesic}
      \VG\st(t) = \VX\st \cos(t\VTH) +\dc\st\sin(t\VTH)  \quad \text{for $t\in [0,1]$}.
  \end{align} 
Note that this geodesic has the length $\fnorm{\VTH}=\darc([\VX],[\VY])$.

We end our discussion  by proving the following property that we will use later. 
\begin{mdframed} 
\begin{lemma}[geodesic convexity] \label{lem:geodesic_convexity}
For a fixed $[\VU]\in\Gr(m,r)$ and $\phi \in [0,\pi/4)$, let $\NN{[\VU]}{\phi}$ be the subset of $\Gr(m,r)$ consisting of points whose principal angles to $[\VU]$ are all less than equal to $\phi$, i.e., letting $\VTH_{[\VX]}$ be the principal angle matrix between $[\VX]$ and $[\VU]$,
\begin{align*}
\NN{[\VU]}{\phi}:= \left\{[\VX]~:~ \VTH_{[\VX]} \preceq \phi \cdot  \VI \right\}.
\end{align*}
Then, $\NN{[\VU]}{\phi}$ is geodesically convex, i.e., for any two points $[\VX],[\VY]\in\NN{[\VU]}{\phi}$, the unique shortest geodesic joining them is entirely contained in  $\NN{[\VU]}{\phi}$.
\end{lemma}
\end{mdframed}

\begin{proof}
We begin the proof by first invoking  the variational characterization of the principal angles (see, e.g., \citep{miao1992principal}): 
\begin{align} \label{exp:variational}
    [\VX] \in \NN{[\VU]}{\phi} \quad \Longleftrightarrow\quad \min_{\substack{\vx \in \col([\VX]),~
    \vu \in \col([\VU])\\
   \norm{\vx}_2=\norm{\vu}_2=1 }}  \left|\inp{\vx}{\vu}\right|  \geq \cos(\phi).
\end{align}
Now pick any two points $[\VX]$ and $[\VY]$ from $\NN{[\VU]}{\phi}$.
Since $\phi < \pi/4 $, it follows from \eqref{exp:variational} that the principal angles between $[\VX]$ and $[\VY]$ are all less than equal to $2\phi <\pi/2$.
Hence, there is a unique geodesic of the shortest length.

For concreteness, let us choose representatives as per \eqref{exp:geodesic_setup}. 
Also, let $\VG\st(t)$ ($t\in[0,1]$) be the unique geodesic of the shortest length between $\VX\st$ and $\VY\st$ as in \eqref{def:geodesic}. 
For simplicity, we may assume that   $\theta_i> 0$ for all $i\in[r]$. In fact, if $\theta_i=0$, the $i$-th column of $\VG\st(t)$ is constantly equal to the $i$-th column of $\VX\st$ along the geodesic.
Hence, we have
 \begin{align*} 
      \VG\st(t) &= \VX\st \cos(t\VTH) +\dc\st\sin(t\VTH) =: \VX\st \VD_1  +\VY\st\VD_2 \,,
  \end{align*}
  where  $\VD_1:= \big( \cot(t\VTH) - \cot(\VTH)\big) \sin(t\VTH)$ and $\VD_2:=[\sin(\VTH)]^{-1}\sin(t\VTH)$ are some diagonal matrices with nonnegative entries (since $\cot(\cdot)$ is decreasing on $[0,\pi/2]$).
 
 Now in order to show that $[\VG\st (t)] \in\NN{[\VU]}{\phi}$, let us choose an arbitrary $\vg \in \col(\VG\st(t)) $. 
 We can write $\vg := \VG\st(t) \vla$ for some $\vla \in \R^{r}$.
 Letting $\vxxx:= \VX\st \VD_1\vla$ and $\vyyy:= \VY\st \VD_2\vla$,
 \begin{align*}
     \vg = \vxxx + \vyyy \quad\text{and}\quad  \inp{\vxxx}{\vyyy} = \vla^\top \VD_1^\top \cos(\VTH) \VD_2 \vla \geq 0\,.
 \end{align*}
 Based on this, we prove the following claim.
 \begin{claim}
 Suppose that $\vx \in \col(\VX)$ and $\vy \in \col(\VY)$ such that $\inp{\vx}{\vy}\geq 0$, $\norm{\vx}_2=\norm{\vy}_2=1$.
 Then,  for any unit-norm vector $\vu \in \col(\VU)$, we have $\inp{\vu}{\vx} \cdot \inp{\vu}{\vy} >0$.
 \end{claim} 
\noindent To prove this claim, suppose to the contrary that there exists a vector $\vu \in \col(\VU)$ such that  $\inp{\vu}{\vx}>0$ but $\inp{\vu}{\vy}<0$.  
Then from \eqref{exp:variational}, it follows that $\inp{\vu}{\vx}> 1/\sqrt{2}$ and 
$\inp{\vu}{\vy} <-1/\sqrt{2}$, since $\phi <\pi/4$.
Then, using the hypothesis $\norm{\vx}_2=\norm{\vy}_2=1$, this implies that 
\begin{align*}
    \sqrt{2} < \inp{\vu}{\vx -\vy} \leq \norm{\vu}_2\norm{\vx -\vy }_2 =\sqrt{2-2\inp{\vx}{\vy}}.
\end{align*}
Hence, it holds that $\inp{\vx}{\vy}<0$.
This  contradicts the hypothesis that $\inp{\vx}{\vy}\geq 0$.

Now due to the claim,  
for any unit-norm   $\vu\in \col(\VU)$, we have either (i) $\inp{\vu}{\vxxx/\norm{\vxxx}_2} \geq  \cos(\phi)$ and $\inp{\vu}{\vyyy/\norm{\vyyy}_2} \geq \cos(\phi)$ or
(ii) $\inp{\vu}{\vxxx/\norm{\vxxx}_2} \leq  -\cos(\phi)$ and $\inp{\vu}{\vyyy/\norm{\vyyy}_2} \leq -\cos(\phi)$.
Since $\vg$ is a unit-norm vector that is a conic combination of two vectors $\vxxx$ and $\vyyy$, it follows that either (i)   $\inp{\vu}{\vg} \geq \cos(\phi)$ or (ii)
$\inp{\vu}{\vg} \leq -\cos(\phi)$. 
In other words, $\left|\inp{\vu}{\vg}\right| \geq \cos(\phi)$ for any unit-norm   $\vu\in \col(\VU)$.
Since $\vg$ was arbitrarily chosen from $\col(\VG\st(t))$, this completes the proof.
\end{proof}

\begin{remark}
When $\phi=\pi/4$, a careful inspection of the proof of Lemma~\ref{lem:geodesic_convexity} reveals that $\NN{[\VU]}{\pi/4}$ is also geodesically convex in an appropriate sense: for any two point  $[\VX],[\VY]\in\NN{[\VU]}{\phi}$, there is at least one geodesic of the shortest length joining them  that is entirely contained in  $\NN{[\VU]}{\phi}$.
However, one can see that this geodesic convexity is no longer true for $\phi>\pi/4$.
For instance, for $\Gr(2,1)$, when $\phi=3\pi/8$, choosing $[\vu] = [(0,1)^\top]$, $[\vx]=[(\cos(\pi/8),\sin(\pi/8))^\top]$ and $[\vy]=[(\cos(-\pi/8),\sin(-\pi/8))^\top]$, one can easily see that the unique geodesic between $[\vx]$ and $[\vy]$ is outside of $\NN{[\vu]}{3\pi/8}$. 
\end{remark}

\begin{remark} \label{rmk:curvature}
From the fact that Grassmann manifolds have sectional curvatures upper bounded by two~\citep{wong1968sectional}, it follows from a classical result in Riemannian geometry~
\citep{klingenberg1959contributions} that the convexity  radius is greater than or equal to $\pi/4$.
In light of this, Lemma~\ref{lem:geodesic_convexity} reveals that in fact  for Grassmannian manifolds, the actual convexity radius is much larger than the classical lower bound.  
\end{remark}

 \section{Matrix factorization over a Grassmann manifold}\label{sec:formulation}

In this section, we describe a formulation of the matrix factorization approach as an optimization over a Grassmannian manifold~\citep{keshavan2010matrix,dai2011subspace,boumal2011rtrmc,dai2012geometric, boumal2015low}. 
First, consider the full observation case, i.e., $\Omega = [m]\times [n]$:
\begin{align} \label{exp:factorization}
    \mini_{\substack{\VX\in\R^{m\times r}\\ 
    \VY\in \R^{n\times r}}}\left[\gf(\VX,\VY):= \frac{1}{2}\fnorm{\VX\VY^\top -\VM}^2\right].
\end{align}
By symmetry, we have 
$\gf(\VX ,\VY )=\gf(\VX^\star \VR,\VY^\star(\VR^{-1})^\top)$ for any invertible $\VR$, and hence without loss generality, we may assume that $\VX$ is orthonormal.
Now in order to reduce the number of optimization parameters, we optimize this cost function over $\VY$.
Denoting $\VYX := \argmin_{\VY\in \R^{n\times r}} \gf(\VX,\VY)$,  the optimality condition yields:
\begin{align*}
    0= \nabla_2\,\gf(\VX,\VYX) =  (\VX\VYX^\top  - \VM)^\top \VX \end{align*}
    which is equivalent to $\VYX \VX^\top \VX =  \VM^\top \VX$.
In particular, since $\VX$ is orthonormal, i.e., $\VX^\top \VX =\VI_r$, we obtain $\VYX = \VM^\top \VX$.
Plugging this back to \eqref{exp:factorization}, we arrive at the following optimization problem over  $\Gr(m,r)$:
\begin{align} \label{cost:formulation_full}
    \mini_{\substack{\VX\in\R^{m\times r}\\ 
    \VY\in \R^{n\times r}}} \left[\ff(\VX):=\frac{1}{2}\fnorm{(\VI_m-\VX\VX^\top)\VM}^2\right].
\end{align}
Note that $\ff(\VX\VQ)=\ff(\VX)$ for any $\VQ\in \orth(r)$ and hence \eqref{cost:formulation_full} is well-defined on $\Gr(m,r)$.
The above derivation can be also found in \citep[\S III-A]{pitaval2015convergence}. 
\begin{remark} \label{rmk:subspace}
A similar objective function was considered in \citep{maunu2019well} under the context of robust subspace recovery.
\end{remark}

To have a glimpse of the cost \eqref{cost:formulation_full}, we illustrate it for the simplest case of rank-$1$ and full observation.
Assume that  $\VM=\sigma \vu\vv^\top $ for unit vectors $\vu$ and $\vv$ and $\sig>0$.
Using $\vx$ to denote the argument of the cost function, \eqref{cost:formulation_full} becomes
\begin{align*}
   \ff(\vx) &= \frac{1}{2}\fnorm{(\VI_m -\vx\vx^\top)\VM}^2=\frac{\sig^2}{2}\fnorm{(\VI_m -\vx\vx^\top)\vu\vv^\top}^2\,.
\end{align*}
See Figure~\ref{fig:illustration} for the landscape for this simplest case.

\begin{figure}
    \centering 
    \includegraphics[width=.4\columnwidth]{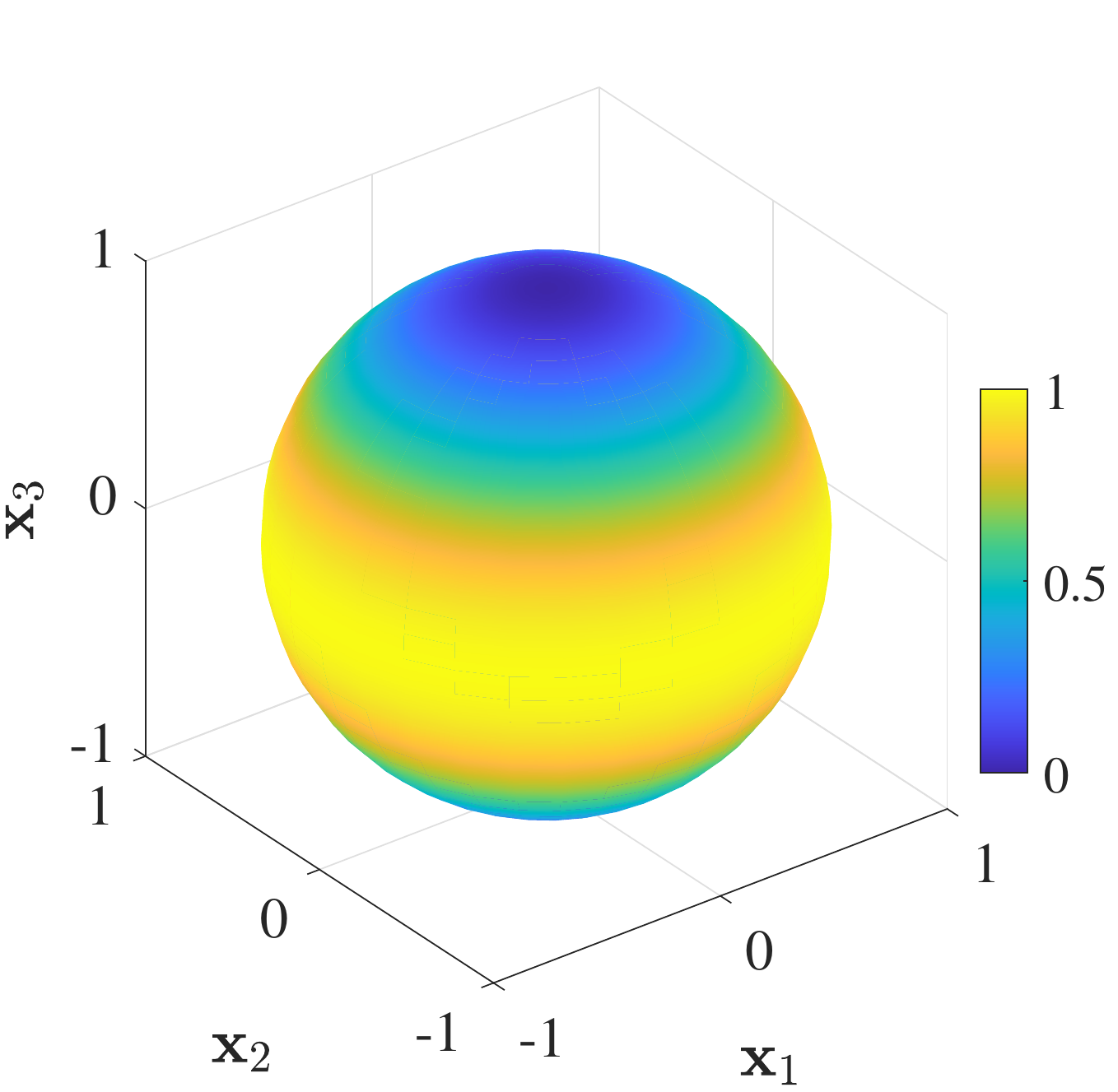} ~~\qquad~~
     \includegraphics[width=.4\columnwidth]{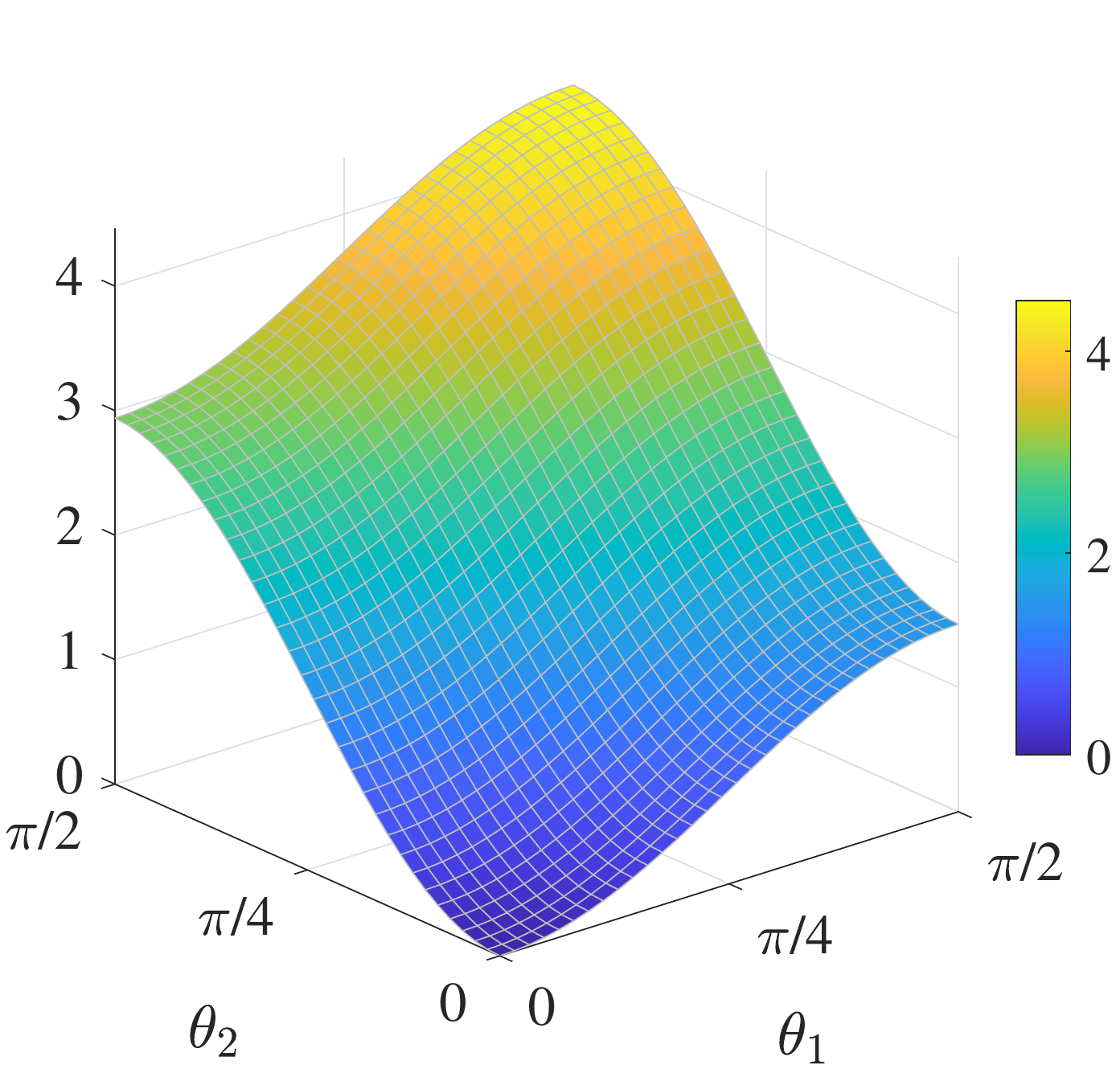}
    \caption{Illustration of the fully observed cost   \eqref{cost:formulation_full}. (Left) The rank-one case with $\sig =\sqrt{2}$ and $\vu = (0,0,1)$. (Right) Illustration of the landscape for the rank-$6$ case using two principal angles to the ground truth $[\VU]$.     We set $m=100$, $n=200$, and $\VS =\diag(1,1,1,1.4,1.4,1.4)$.
 }
    \label{fig:illustration}
\end{figure}

In general, when only partially observed entries are available, we have
\begin{align*} 
    \mini_{\substack{\VX\in\R^{m\times r}\\ 
    \VY\in \R^{n\times r}}}\left[\gp(\VX,\VY):= \frac{1}{2}\pnorm{\VX\VY^\top -\VM}^2\right].
\end{align*}
Similarly to the fully observed case, one can reduce the number of parameters by considering  $\VYX := \argmin_{\VY\in \R^{n\times r}} \gp(\VX,\VY)$. Then, we obtain the optimization problem
\begin{align*} 
    \mini_{\VX\in\R^{m\times r},~\VX^\top\VX=\VI } \left[\fp(\VX):= \min_{\VY\in \R^{n\times r}}\gp(\VX, \VY) \right].
\end{align*}

In the subsequent section, we first focus on understanding the fully observed case.
We will continue our discussion regarding the partially observed case in Section~\ref{sec:partial}.

\subsection{Interpretation based on the projection distance} \label{sec:projection}

To better understand the formulation, we provide a fruitful interpretation for the full observation cost \eqref{cost:formulation_full}.
Our interpretation will be based on  the principal angle and the  distances introduced in Section~\ref{sec:principal}. 
For setup,  we make the following choice of representative elements in the equivalence class, motivated by Proposition~\ref{prop:diagonal}. 
We call this choice the \emph{principal alignment}, and we will repeatedly use this choice in later calculations: 
\begin{mdframed}[roundcorner=10pt,shadow=false] 
\begin{custom3}[\hypertarget{setup}{{\sf Principal alignment}}]
Let $\VTH$ be the principal angle matrix  between $[\VX]$ and $[\VU]$.
\begin{enumerate}
    \item By Proposition~\ref{prop:diagonal},  there exist   $\VX\st \in [\VX]$, $\VU\st \in [\VU]$ s.t. $\VX\st^\top \VU\st = \cos (\VTH)$.
    \item Let $\dd\in \T{\VX\st}\Gr(m,r)$ be an orthonormal matrix s.t. $\dd \sin (\VTH)= (\VI_m-\VX\st\VX\st)^\top\VU\st $.
\end{enumerate}
Write $\VU\st =\VU\VQ$ for $\VQ\in \orth(r)$.
Let $\QQ := \VQ^\top \VS^2 \VQ\in \R^{r\times r}$ and $\qq_i:=\QQ_{i,i}$ for $i\in[r]$. 
Note that $\smin^2 \leq \qq_i \leq \smax^2$ for all $i\in[r]$.
\end{custom3} 
\end{mdframed}
\begin{remark} $\VTH,\dc\st, \VQ,\QQ$ all depend on the point $[\VX]$.
\end{remark}
With the  \setup, it is easy to see that
\begin{align*}
    \ff(\VX) &= \frac{1}{2}\fnorm{(\VI_m -\VX\st\VX\st^\top )\VU\st \VQ^\top \VS\VV^\top}^2\\
    &= \frac{1}{2}\fnorm{ \dd \sin(\VTH) \VQ^\top \VS}^2= \frac{1}{2}\Tr[\sin^2 (\VTH)\QQ ]\,.
\end{align*}
See Figure~\ref{fig:illustration} for an illustration of the landscape based on the principal angles to $[\VU]$.
Therefore, we arrive at the following  interpretation:
\begin{center}
The cost $\ff$ is equal to a weighted  projection distance.
\end{center} 
Having this interpretation, we now study the geometry of the cost function.

\section{Geometry of the fully observed  case} \label{sec:landscape}

 In this section, we study the geometry of the full observation cost $\ff$ defined in  \eqref{cost:formulation_full}. 
 Let us first compute the derivatives. For any $\dc \in \T{\VX}\Gr(m,r)$, 
    \begin{align}
      \inp{\grad \ff(\VX)}{\dc}&= D \ff(\VX)[\dc]   \nn\\
       &= - \frac{1}{p}\inpp{  (\dc \VX^\top+ \VX \dc^\top) \VM }{ (\VI_m -\VX\VX^\top)\VM  }\,. \label{grad:full}\\
    \hess \ff(\VX)[\dc,\dc]&=D(\grad \ff(\VX)[\dc])[\dc] \nn\\
    \begin{split}
        &= - \frac{2}{p}\inpp{  \dc \dc^\top \VM }{(\VI_m -\VX\VX^\top)\VM }  +\frac{1}{p}\pnorm{  (\dc \VX^\top+ \VX \dc^\top) \VM }^2\,. 
    \end{split} \label{hess:full}
\end{align}

  \subsection{Geometry based on gradient} \label{sec:full_gradient}

We first derive a compact expression for gradient based on the principal angles.   
 
  \begin{theorem}[characterization of gradient] \label{thm:gradient}
  Assume  $\Omega=[m]\times[n]$.
  Let $[\VX]\in\Gr(m,r)$.
 With the \setup,  the gradient has the following compact expression:
\begin{align} \label{grad:full_compact}
    \grad \ff(\VX\st)  =  - \dd \sin(\VTH)\,\,\QQ\,\cos (\VTH)\,.
\end{align}
Consequently, $ \frac{\smin^4}{4}\fnorm{\sin (2\VTH)}^2\leq  \fnorm{\grad \ff([\VX])}^2 \leq \frac{\smax^4}{4}\fnorm{\sin (2\VTH)}^2$.
 \end{theorem}
\begin{proof}
Rewriting the first derivative \eqref{grad:full} under the \setup, $\inp{\grad \ff(\VX\st)}{\dc}$ can be written as
\begin{align*}
     - \inp{  (\dc \VX\st^\top+ \VX\st \dc^\top) \VM }{ (\VI_m -\VX\st\VX\st^\top)\VM)  }\,.
\end{align*}
On the other hand, using the fact that $(\VI_m -\VX\st\VX\st^\top)\VX\st =\VZE$, the second term in the above expression vanishes: $\inp{  \VX\st \dc^\top \VM } {(\VI_m- \VX\st \VX\st^\top) \VM }= \inp{(\VI_m- \VX\st \VX\st^\top)\VX\st \dc^\top \VM}{   \VM  }= 0$. 
    Thus, after rearranging, we obtain the desired compact expression: 
    \begin{align*}
        \grad \ff(\VX\st)  &= -  (\VI_m- \VX\st \VX\st^\top) \VM\VM^\top \VX\st  \\
        &= -  (\VI_m- \VX\st \VX\st^\top) \VU\st \VQ^\top \VS^2 \VQ \VU\st^\top\VX\st \\
        &= - \dd \sin(\VTH)\,\,\QQ\,\cos (\VTH)\,.
    \end{align*}
Now,  the last part of the theorem follows due to the definition of $\QQ$  together with the basic trigonometric fact $\sin(2\theta) = 2 \sin(\theta) \cos (\theta)$.
\end{proof}

We remark that a similar calculation is done in \citep[\S III]{pitaval2015convergence}.
With Theorem~\ref{thm:gradient}, the following characterization of critical points is immediate.
 \begin{mdframed}[backgroundcolor=emph]
  \begin{corollary}[critical points] \label{cor:criticalpoints}   Assume  $\Omega=[m]\times[n]$.
  Let $[\VX]$ be a point on $\Gr(m,r)$, and $ \theta_1,\dots,\theta_r$ be the principal angles between $[\VX]$ and $[\VU]$.
  Then, $[\VX]$ is a critical point of $
  \ff$ if and only if  $\theta_i=0$ or $\pi/2$ $\forall i\in [r]$. 
 \end{corollary}
 \end{mdframed}

 \subsection{Geometry based on Hessian} \label{sec:full_hessian}
 
 We first derive a compact expression for Hessian based on the principal angles.   

\begin{theorem}[characterization of the Hessian]  \label{hessian}
 Assume $\Omega=[m]\times[n]$. Let $[\VX]\in\Gr(m,r)$.
 With the \setup,  the second derivative $\hess \ff (\VX\st)[\dc,\dc]$  has the following compact expression: for any direction  $\dc \in \T{\VX\st}\Gr(m,r)$, 
  \begin{align} \label{hess:full_compact}
  \begin{split}
    &\Tr\left[  \cos (\VTH) \dc^\top \dc  \cos(\VTH) \QQ\right] - \Tr\left[ \sin (\VTH)  \dd^\top\dc  \dc^\top\dd \sin (\VTH) \QQ \right].  
  \end{split}
  \end{align}
\end{theorem}

\begin{proof} 
Rewriting the second derivative \eqref{hess:full} under the \setup, 
\begin{align*} 
   &- 2\inp{  \dc \dc^\top \VM }{(\VI_m -\VX\st\VX\st^\top)\VM } +\fnorm{  (\dc \VX\st^\top+ \VX\st \dc^\top) \VM }^2.  
\end{align*}
We will simplify the above two terms one by one.

First, from the \setup, note that $(\VI_m -\VX\st \VX\st^\top)\VM =  \dd \sin (\VTH) \VQ^\top\VS \VV^\top$.
Hence, using the fact that $(\VI_m -\VX\st\VX\st^\top)\dc =\dc$, we have
\begin{align}
 \inp{  \dc \dc^\top \VM }{(\VI_m -\VX\st\VX\st^\top)\VM }  
&= \inp{  \dc \dc^\top  }{(\VI_m -\VX\st\VX\st^\top)\VM \VM^\top (\VI_m -\VX\st\VX\st^\top)} \nn  \\
&= \inp{  \dc \dc^\top }{\dd \sin (\VTH)\, \QQ \, \sin (\VTH) \dd^\top }\nn\\
&= \Tr\left[ \Big( \sin (\VTH)\, \dd^\top\dc \dc^\top\dd\, \sin (\VTH)\ \QQ  \Big)\right]\,. \label{trace:term} 
\end{align}
Hence, it follows that the first term is equal to
$-2 \Tr\left[ \Big( \sin (\VTH)\ \dd^\top\dc \dc^\top\dd\ \sin (\VTH)\, \QQ  \Big)\right]$.

Next we calculate the second term.
First, expanding out the Frobenius norm,  we obtain
\begin{align*}  
\fnorm{  \dc \VX\st^\top \VM }^2 +2\inp{\dc \VX\st^\top \VM}{\VX\st \dc^\top \VM} + \fnorm{ \VX\st \dc^\top \VM }^2 \,.
\end{align*}
Note that the middle term  is $0$ since 
$(\VI_m-\VX\st\VX\st^\top)\VX\st =\VZE$. 
Moreover,  one can easily check that the third term is equal to  \eqref{trace:term}:
\begin{align*}
    \fnorm{ \VX\st \dc^\top \VM }^2=\fnorm{   \dc^\top \VM }^2= \inp{ \dc \dc^\top \VM }{ \VM\VM^\top } =\eqref{trace:term}.
\end{align*}
Therefore, the three terms above is equal to
\begin{align*} \fnorm{  \dc \VX\st^\top \VM }^2 + \Tr\left[ \Big( \sin (\VTH)\, \dd^\top\dc \dc^\top\dd\, \sin (\VTH)\, \QQ  \Big)\right]\,.
\end{align*}
Now, let us calculate $\fnorm{  \dc \VX\st^\top \VM }^2$. Using the definition of $\QQ$, we have $\inp{\dc\VX\st^\top \VM}{\dc\VX\st^\top \VM}     =\inp{\VX\st\dc^\top \dc\VX\st^\top}{\VM\VM^\top} = \inp{\VX\st\dc^\top \dc\VX\st^\top}{\VU\st \QQ \VU\st^\top}  = \Tr\left[\cos (\VTH)\,\dc^\top \dc\, \cos(\VTH)  \,  \QQ  \right]$.

Combining the above calculations, we obtain  the desired expression.
\end{proof}

With the compact expression for Hessian, one can characterize the landscape of the cost function.
See Figure~\ref{fig:landscape} for illustrations of the landscape.

\begin{mdframed}[backgroundcolor=emph]
\begin{corollary}[landscape based on Hessian] \label{cor:landscape}
 Assume  $\Omega=[m]\times[n]$.
  Let $[\VX]\in\Gr(m,r)$, and $ \theta_1,\dots,\theta_r$ be the principal angles between $[\VX]$ and $[\VU]$. 
  \begin{enumerate}
      \item If $\theta_i \in [0,\pi/4]$ $\forall i\in[r]$, then $\hess \ff ([\VX])[\dc,\dc]$ is nonnegative   for any $\dc \in \T{[\VX]}\Gr(m,r)$.
      \item If $\theta_i >\frac{\pi}{4}$ for some $i\in[r]$, then there exists a direction $\dc \in \T{[\VX]}\Gr(m,r)$ along which the second derivative is negative.
  \end{enumerate}
\end{corollary}
\end{mdframed}
\begin{proof}
To parse the expression for the second derivative \eqref{hess:full_compact}, we rewrite the direction  $\dc\in\T{\VX\st}\Gr(m,r)$,
as $\dcn \sin \VPH$, where $\dcn$ is orthonormal, i.e., $\dcn^\top \dcn =\VI_r$ and $\VPH=\diag(\phi_1,\dots, \phi_r)\in \R^{r\times r}$ is a diagonal matrix with entries in $[0,\pi/2]$.
Then, the second derivative \eqref{hess:full_compact} becomes
\begin{align} \label{hess:full_angle}
\begin{split}
    &\Tr\left[  \sin^2(\VPH) \cos^2(\VTH) \QQ\right]  - \Tr\left[\sin (\VTH) \dd^\top\dcn \sin^2(\VPH) \dcn^\top\dd  \sin (\VTH)\QQ \right].
\end{split}
\end{align}
Having this expression, we consider the two cases in the statement separately:
\begin{enumerate}
    \item First suppose that $\theta_i \in [0,\pi/4]$ $\forall i\in[r]$.  
    Then, we have $\sin(\theta_i)   \leq 1/\sqrt{2} \leq  \cos(\theta_i)$ $\forall i\in[r]$.
From this, one can see that    \begin{align*}
        \sin (\VTH)\ \dd^\top\dcn\ \sin^2(\VPH)\ \dcn^\top\dd \ \sin (\VTH) \preceq 1/2   \sin^2(\VPH).
    \end{align*}
    On the other hand,
    since $\QQ$ is a PSD matrix, 
    \begin{align*}          \eqref{hess:full_angle} \geq \Tr[\sin^2(\VPH)\ (\cos^2(\VTH)-1/2\VI_r)\QQ]\geq 0.
    \end{align*}
     
    \item Next, suppose that  $\theta_i>\frac{\pi}{4}$ for some $i\in[r]$. Then, $\cos^2(\theta_i)-\sin^2(\theta_i) <0$.
    Now, choose $\dcn=\dd$ and $\VPH$ so that $\phi_i=\pi/2$ and $\phi_j=0$ for all $j\neq i$. With such a choice,  \eqref{hess:full_angle} becomes $\qq_i(\cos^2(\theta_i)-\sin^2(\theta_i)) <0$.
\end{enumerate}
This completes the proof of Corollary~\ref{cor:landscape}.
\end{proof}

\begin{figure} 
    \centering 
    \includegraphics[width=.6\columnwidth]{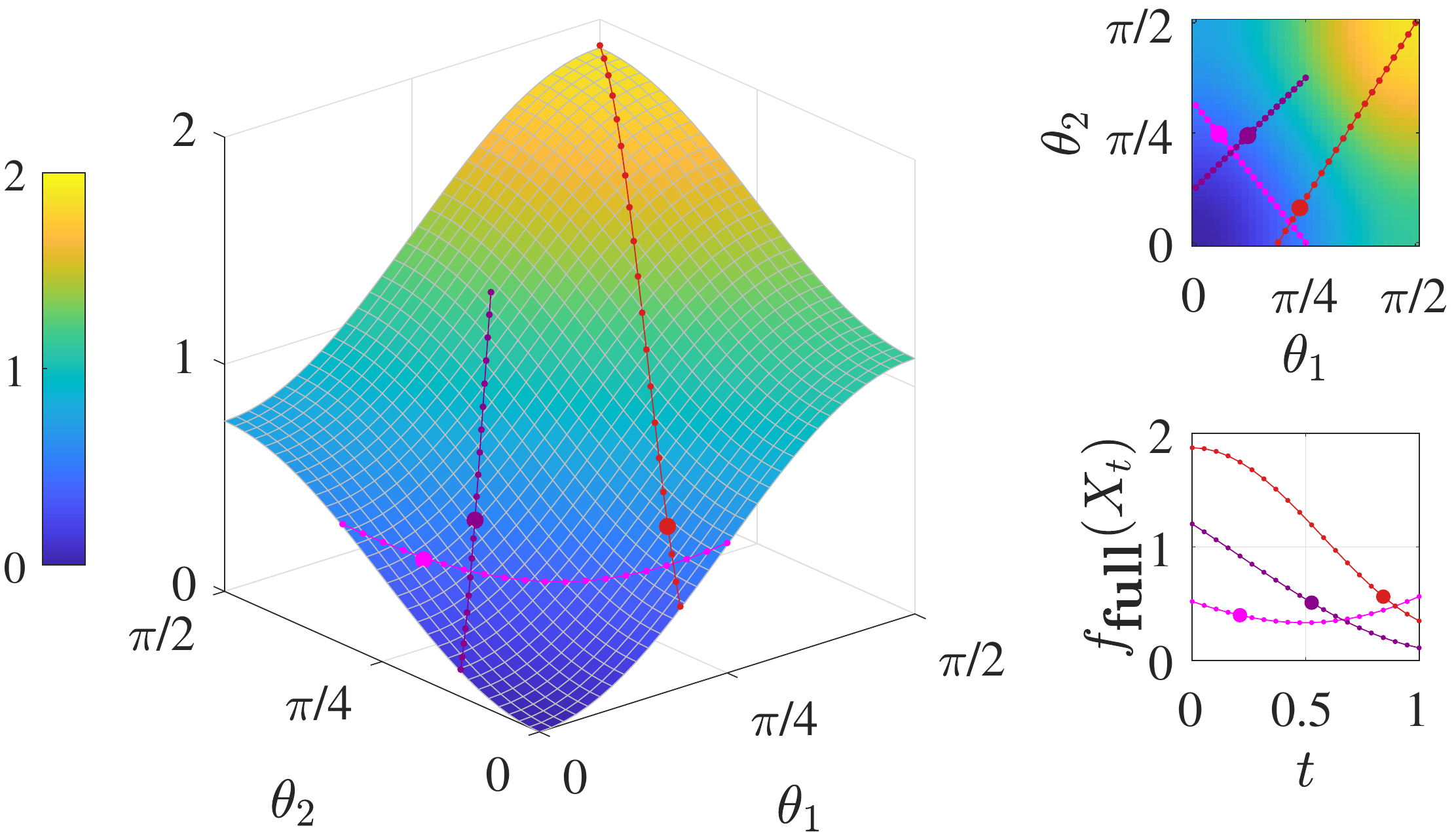}
    \caption{Illustrations of the landscape  of cost function $\ff$. We choose the same parameters as in Figure~\ref{fig:illustration}. We overlay three geodesics with bigger dots that represent the points where the geodesics  enter the region $\NN{[\VU]}{ \pi/4}$. }
    \label{fig:landscape}
\end{figure}

In fact, the proof reveals that if $\theta_1,\dots, \theta_r \in [0,\phi]$ for some $\phi<\pi/4$, the second derivative is greater than equal to $\min_{i\in[r]} [\qq_i\cdot (\cos^2(\phi)-\sin^2(\phi))]$.
This together with the geodesic convexity of $\NN{[\VU]}{\phi}$  (Lemma~\ref{lem:geodesic_convexity}) shows that the cost $\ff$ is strongly  geodesically-convex.

\begin{mdframed}[backgroundcolor=emph]
\begin{corollary}[geodesic convexity] \label{cor:strongconvex}
 Assume   $\Omega=[m]\times[n]$ and  $\phi\in [0,\pi/4)$.
 Following the notations from Lemma~\ref{lem:geodesic_convexity}, for any $[\VX], [\VY]\in \NN{[\VU]}{\phi}$, we have  $$
     \ff(\VY) \geq \ff(\VX) + \inp{\grad \ff(\VX)}{(\VI_m-\VX\VX^\top)\VY} + \frac{\mu}{2}\darc(\VX,\VY)^2,
     $$
     where   $\mu:= \smin \cdot (\cos^2(\phi)-\sin^2(\phi))$.
 In other words, $\ff$ is $\mu$-strongly geodesically-convex in  $\NN{[\VU]}{\phi}$.
  Moreover, $\ff$ is geodesically-convex in  $\NN{[\VU]}{\pi/4}$.
\end{corollary}
\end{mdframed}

Next,   Corollaries~\ref{cor:criticalpoints} and \ref{cor:landscape} together conclude that  all critical points are strict saddle points.
In other words, there is no spurious local minima for $\ff$.
We formally write this conclusion below.
Note that this conclusion  is  consistent with the  previously results in the Euclidean domain~\citep{ge2016matrix,ge2017no}.

\begin{mdframed}
\begin{corollary}[escaping direction] \label{cor:escaping}
Assume  $\Omega=[m]\times[n]$.
Suppose that  $[\VX] \in \Gr(m,r)$  is a critical point of $\ff$, and let $\VTH=\diag(\theta_1,\dots,\theta_r)$ be the principal angle matrix between $[\VX]$ and $[\VU]$.
Then the following hold:
\begin{itemize}
    \item (Corollary~\ref{cor:criticalpoints}) Each $\theta_i$ is equal to either $0$ or $\pi/2$.
    \item (Corollary~\ref{cor:landscape}) With the \setup, the direction $\dcc\in \T{\VX\st}\Gr(m,r)$ defined as $\VU\st \sin( \VTH)$ (i.e., the $i$-th column of $\dcc$ equals that of $\VU\st$ if $\theta_i=\pi/2$ and $\vze$ if $\theta_i=0$) satisfies $\hess \ff (\VX\st)[\dcc,\dcc] = -\Tr(\sin^2(\VTH)\QQ)<0$. 
\end{itemize} 
\end{corollary}
\end{mdframed} 
\begin{proof} 
The first statement is the restatement of Corollary~\ref{cor:criticalpoints}.
For the second statement, first note from \eqref{hess:full_angle} that $\hess \ff ([\VX])[\dc\st\sin(\VTH),\dc\st\sin(\VTH)]$ is equal to
\begin{align*}
     \Tr \left[  \Big(\sin^2(\VTH) (\cos^2(\VTH) 
 - \sin^2 (\VTH))\Big)\QQ  \right]
 =-\Tr \left[  \sin^2(\VTH)\QQ  \right],
\end{align*}
where the equality follows from the fact that each $\theta_i$ is equal to either $0$ or $\pi/2$.
From the \setup, $\dd \sin (\VTH)= (\VI_m-\VX\st\VX\st)^\top\VU\st  =\VU\st \sin (\VTH)$, where again the equality holds since each $\theta_i$ is equal to either $0$ or $\pi/2$. Hence, in fact $\dd \sin (\VTH) =\dcc$.
\end{proof}

 \section{Geometry of the partially observed  case} \label{sec:partial}
  
In this section, we empirically study the geometry of the partially observed case based on our findings in Section~\ref{sec:landscape}.
For simplicity, we focus on  a simple \emph{uniform} observation model where each entry of $\VM$ is observed independently with probability $p \in (0,1]$. 
  
  We recall the formulation for the partially observed case.
\begin{align} \label{cost:formulation} 
    \mini_{\VX\in\R^{m\times r},~\VX^\top\VX=\VI } \Big[ \fp(\VX)&:= \frac{1}{p}\cdot \min_{\VY\in \R^{n\times r}}\gp(\VX, \VY)\Big],\\
 \text{where}~~   \gp(\VX,\VY)&:= \frac{1}{2}\pnorm{\VX\VY^\top -\VM}^2. \nn
\end{align}
In fact, we multiply the cost by $1/p$ for a correct scale.
We first empirically study the landscape of this formulation. Later in this section,  we discuss some variants of this formulation considered in previous works.

\subsection{Landscape simulations}

\paragraph{Settings.}  We   follow the setups considered in \citep[\S5]{boumal2015low}. 
For all setups, orthonormal matrices 
$\VU\in \R^{m\times r}$ and $\VV\in \R^{n\times r}$ are generated uniformly at random.
\begin{itemize}
    
     \item \emph{Setting 1: rectangular matrices.} We set $m=1000$, $n=30\,000,~r=6$ and $\VS = \diag(1,1+\tfrac{1}{5},\cdots,1+\tfrac{4}{5},2)$.  
     
    \item \emph{Setting 2: high dimension.} We set $m=n= 10\,000,~r=10$ and   $\VS = \diag(1,1+\tfrac{1}{9}, \dots, 1+\tfrac{8}{9}, 2 )$.

    \item \emph{Setting 3: bad conditioning.} 
    We set $m=n=1000,~r=10$. To make the ground truth matrix ill-conditioned, we set $\VS = \sqrt{mn}\cdot  \diag(1,e^{5/9}, e^{10/9},\dots,  e^{40/9}, e^{5})$. 
   
\end{itemize}

\begin{figure}
    \centering  
    \includegraphics[width=.49\columnwidth]{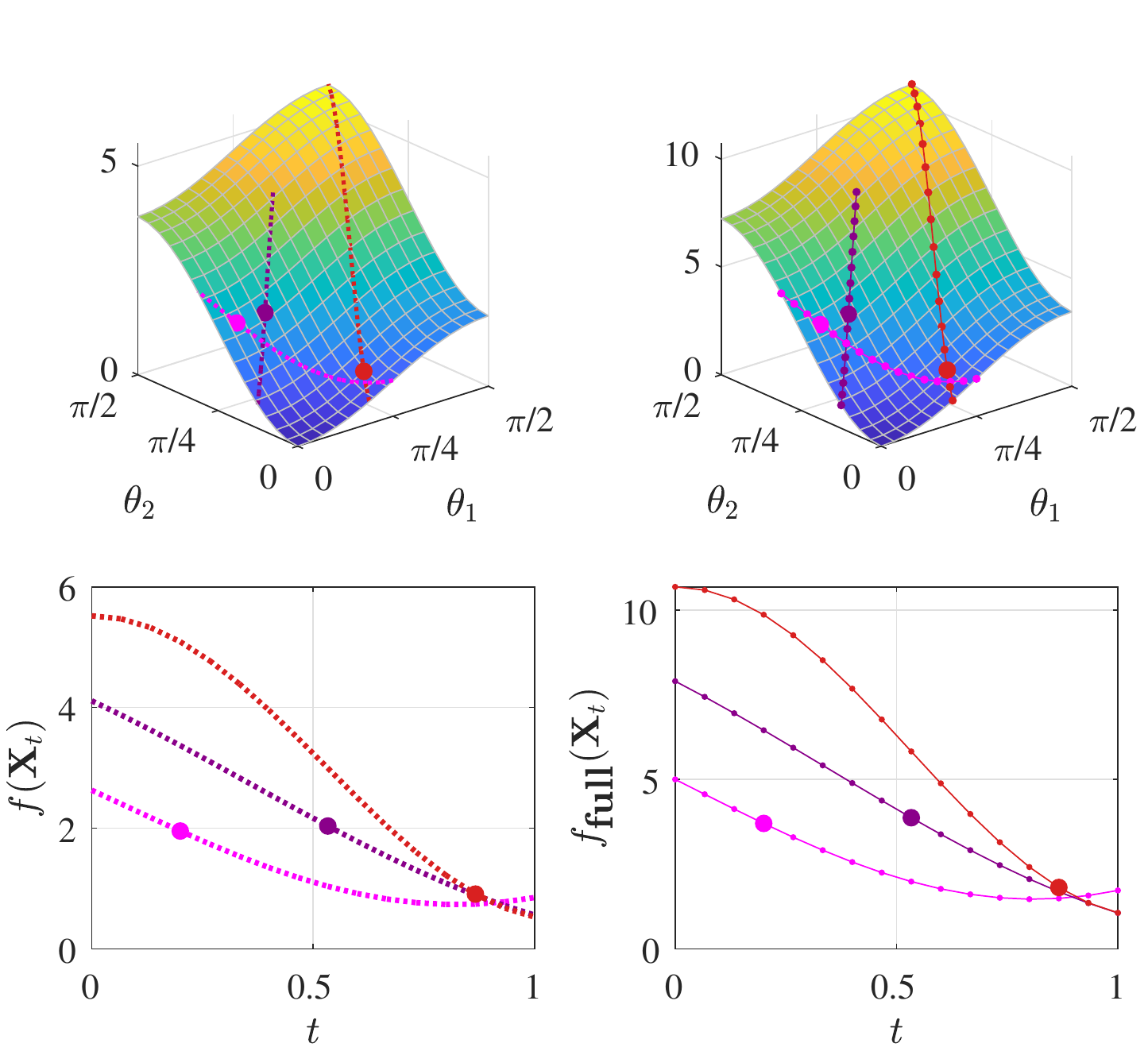} 
    \includegraphics[width=.49\columnwidth]{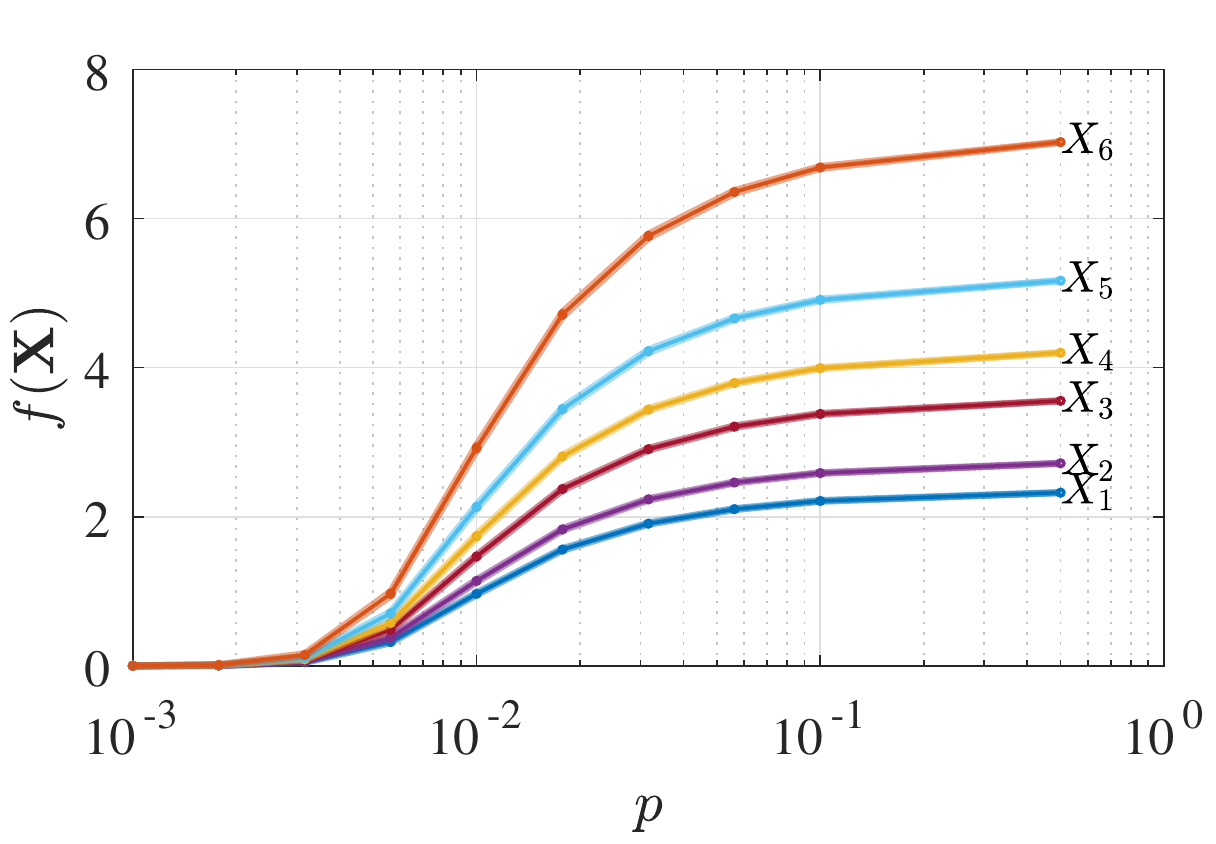} 
    \includegraphics[width=.49\columnwidth]{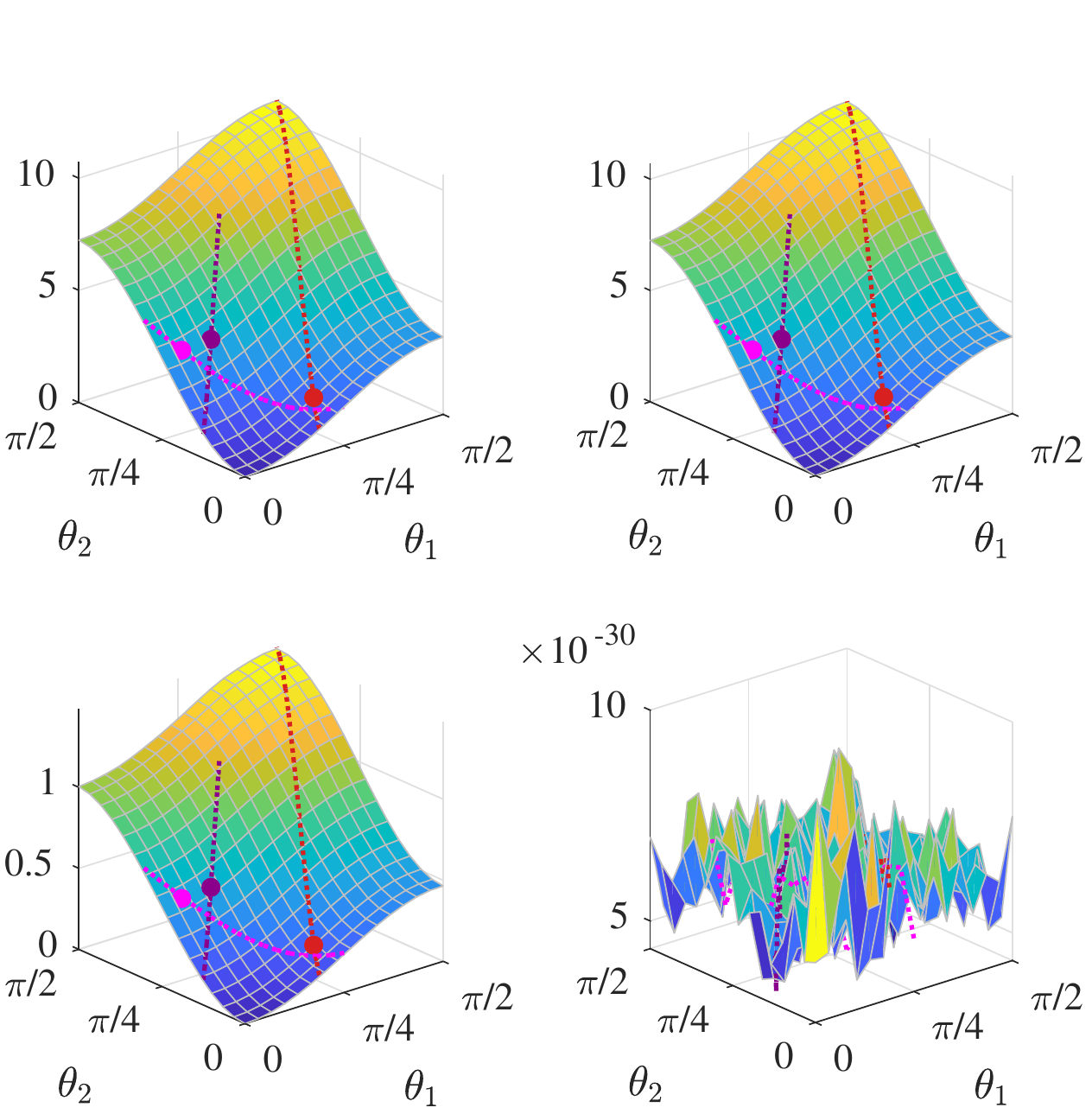}
    \caption{Simulation results for Setting 1. (Top-left) Results for Experiment 1.   Note that the landscape ($p=0.026$) is very similar to the fully observed case ($p=1$). (Top-right) Results for Experiment 2. We depict the error bars with shades around each curve.
    (Bottom) Results for Experiment 3.  Note that for  $p=0.1,~0.01$ the landscapes are quite similar to the fully observed case ($p=1$).   }
    \label{fig:sim:sc1}
\end{figure}

\begin{figure}
    \centering  
    \includegraphics[width=.49\columnwidth]{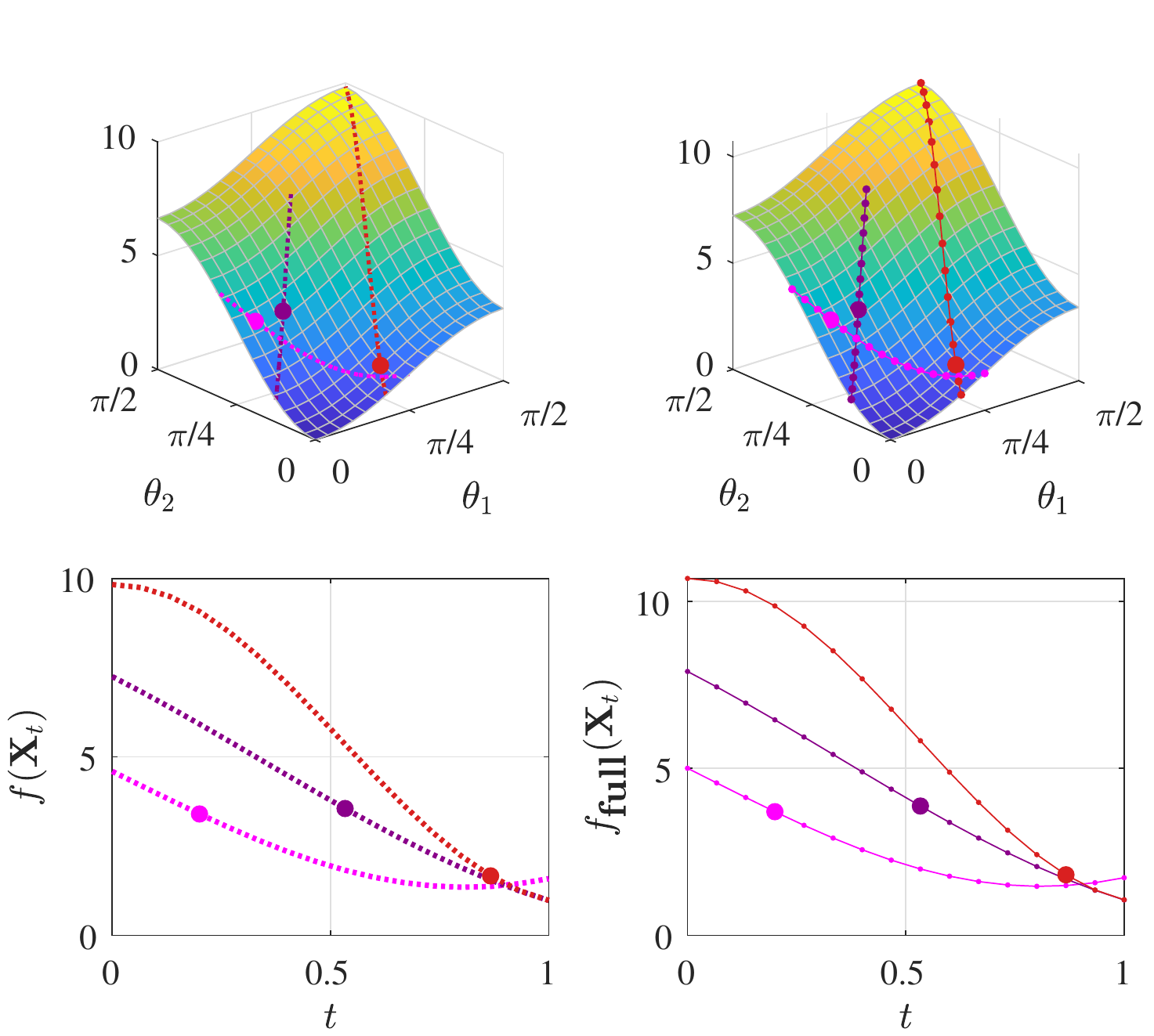} 
    \includegraphics[width=.49\columnwidth]{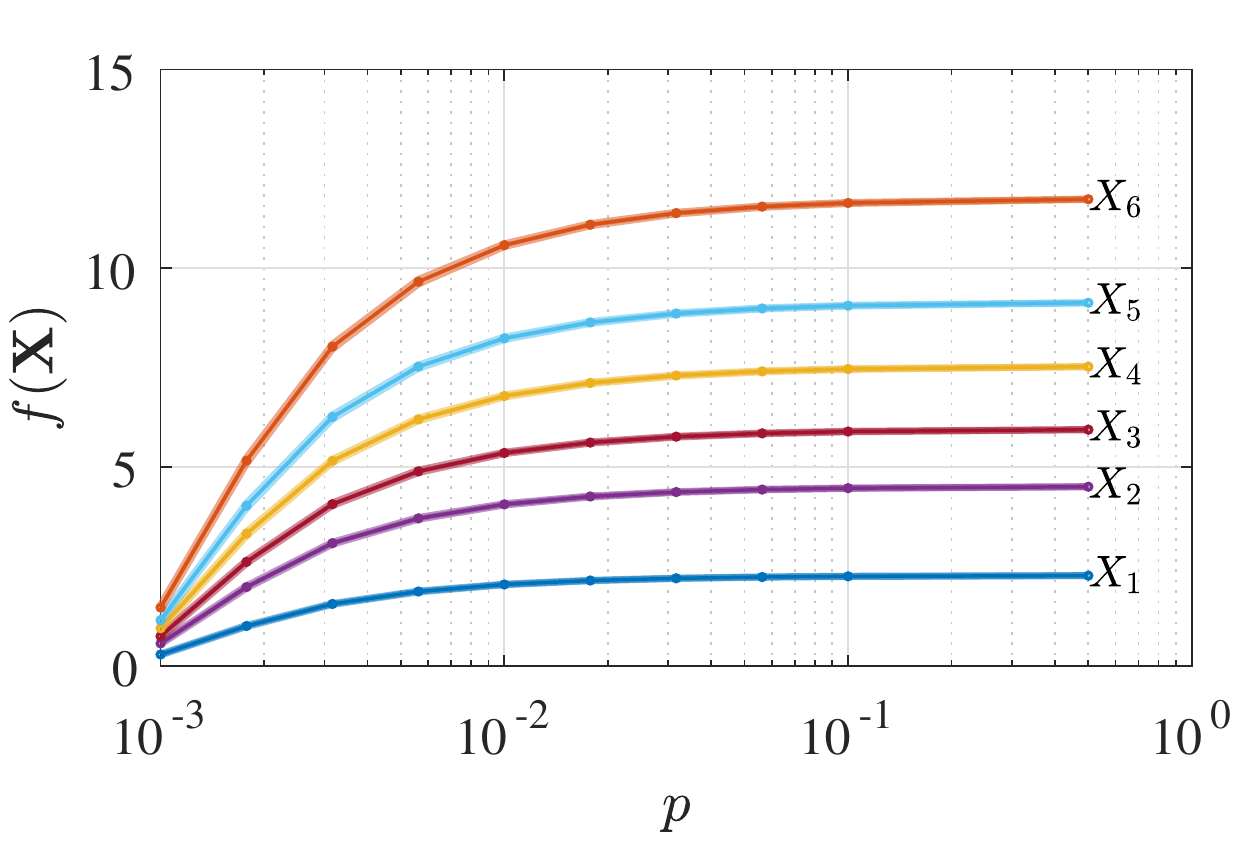} 
    \includegraphics[width=.49\columnwidth]{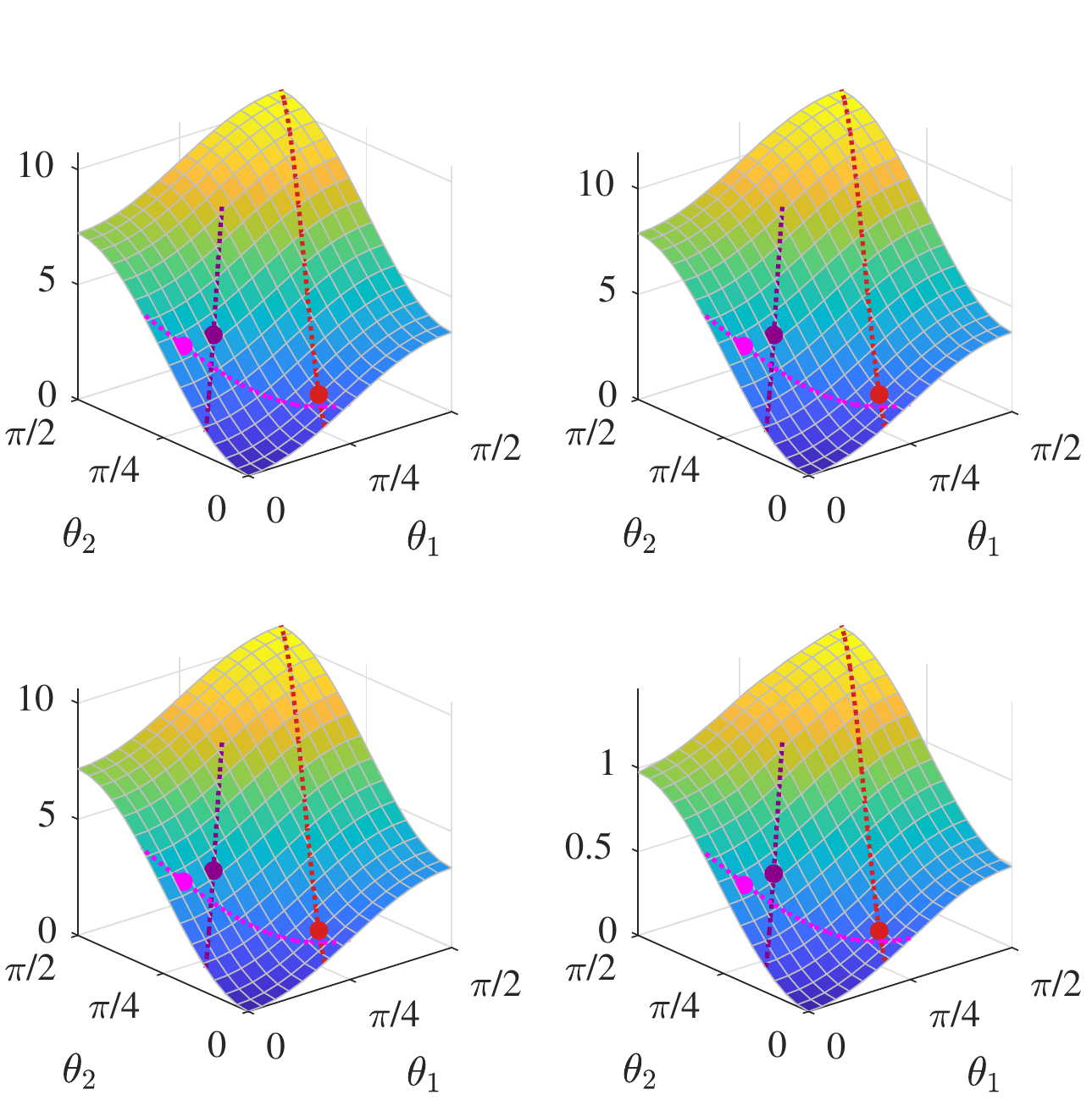}
    \caption{Simulation results for Setting 2. (Top-left) Results for Experiment 1.   Note that the landscape ($p=0.006$) is very similar to the fully observed case ($p=1$). (Top-right) Results for Experiment 2. We depict the error bars with shades around each curve.
    (Bottom) Results for Experiment 3.  Note that the landscapes are quite similar to the fully observed case ($p=1$). }
    \label{fig:sim:sc2}
\end{figure}

\begin{figure}
    \centering  
    \includegraphics[width=.49\columnwidth]{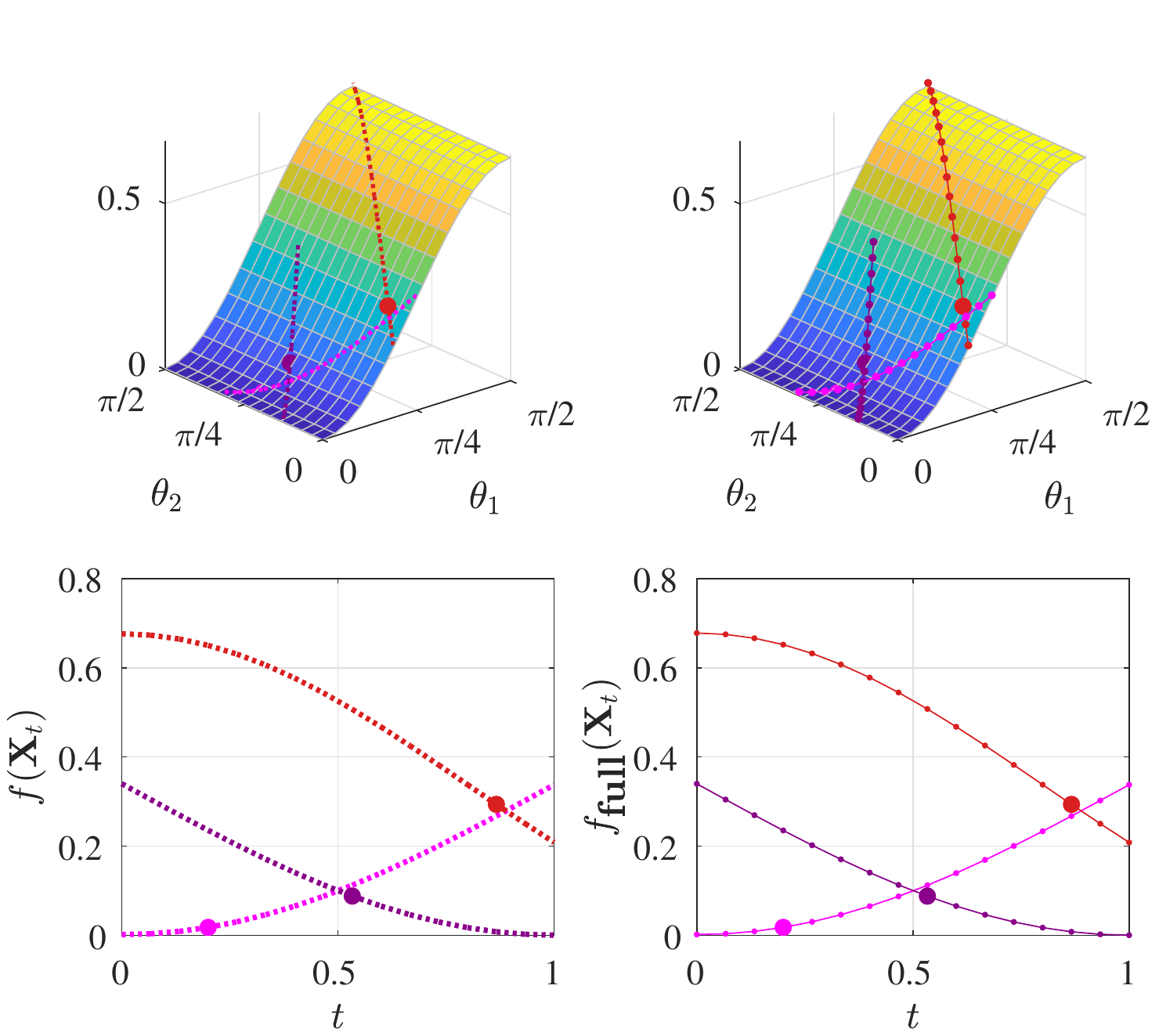} 
    \includegraphics[width=.49\columnwidth]{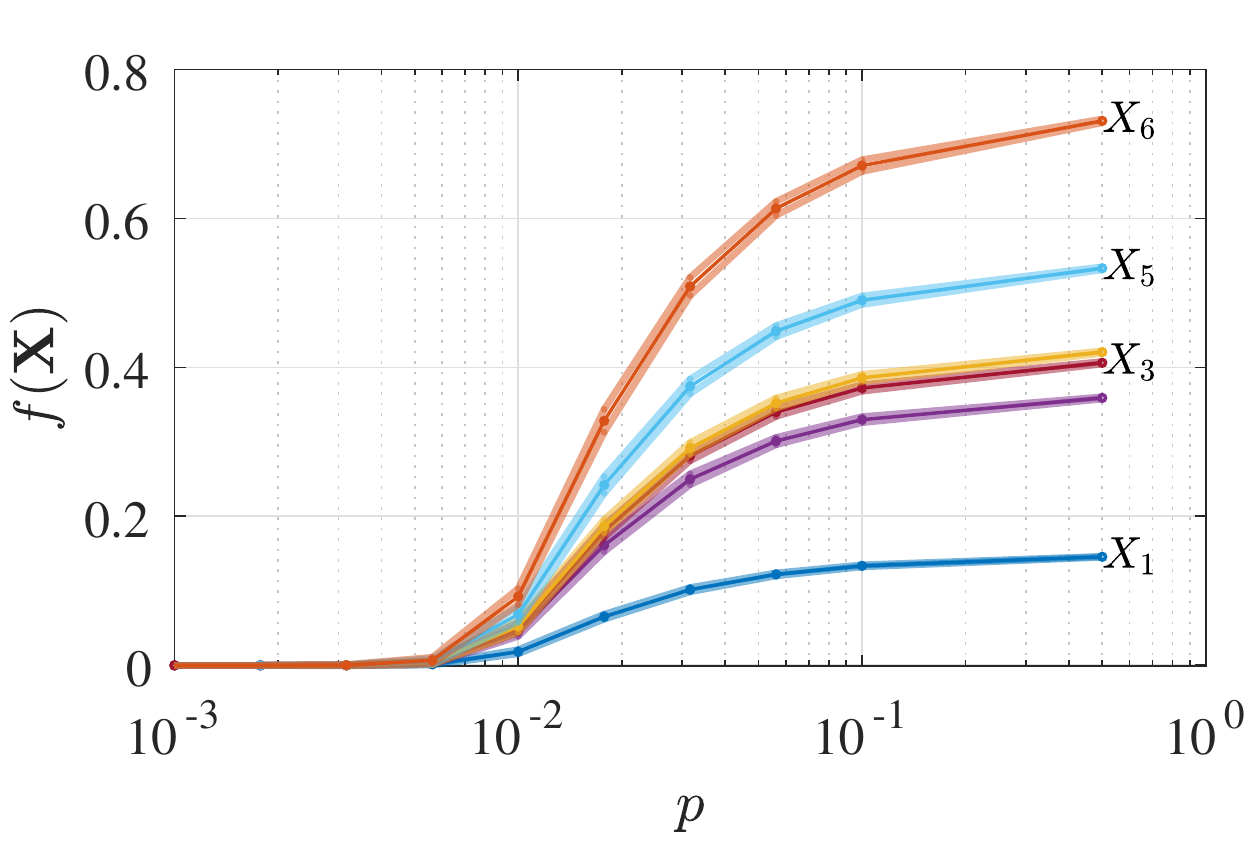} 
    \includegraphics[width=.49\columnwidth]{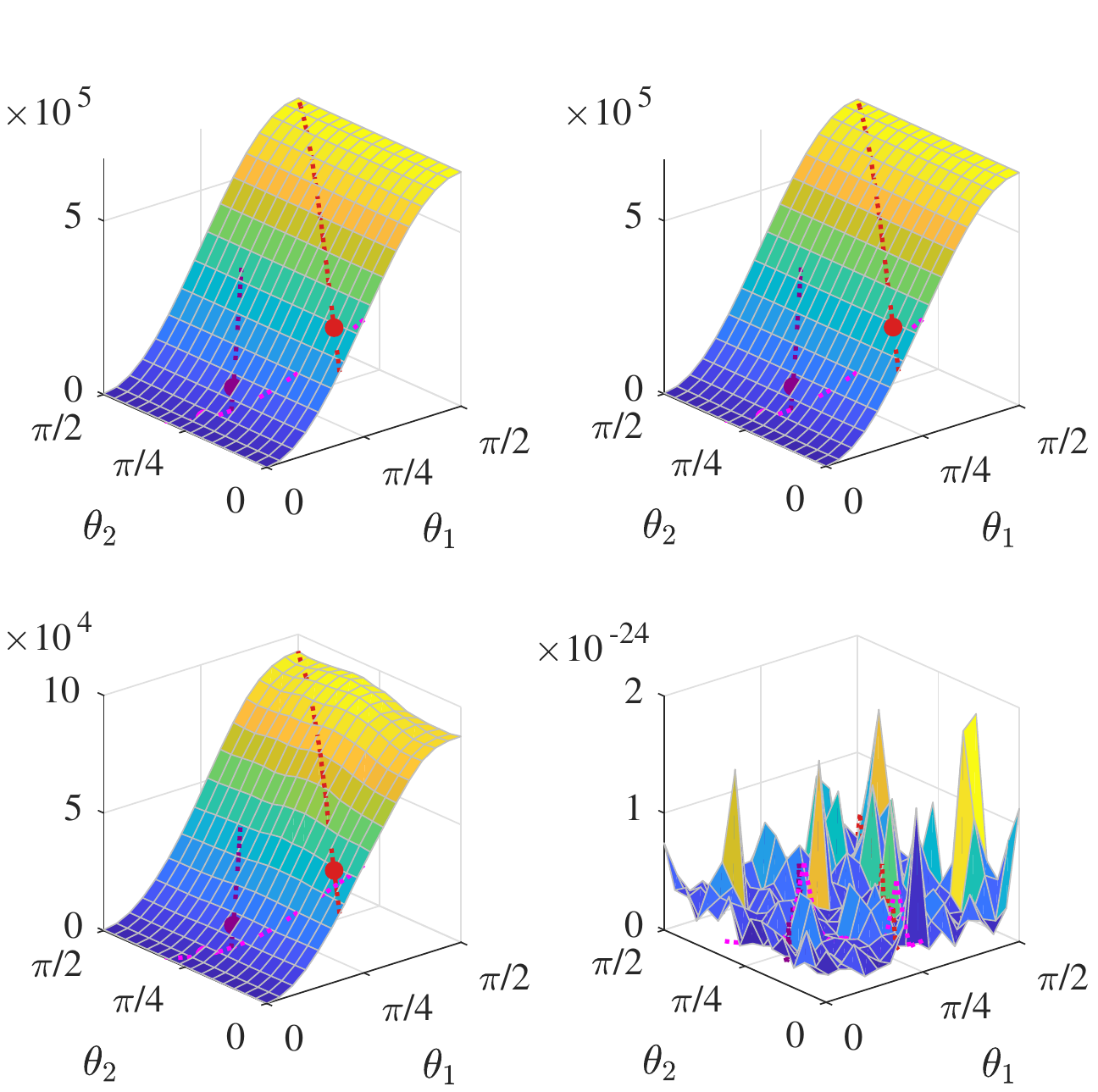}
    \caption{Simulation results for Setting 3.  (Top-left) Results for Experiment 1.   Note that the landscape ($p=0.1$) is very similar to the fully observed case ($p=1$). (Top-right) Results for Experiment 2. We depict the error bars with shades around each curve.
    (Bottom) Results for Experiment 3.  Note that for  $p=0.1,~0.01$ the landscapes are quite similar to the fully observed case ($p=1$). }
    \label{fig:sim:sc3}
\end{figure}
 
\paragraph{Simulations.}
Under the above scenarios, we run the following experiments. See Figures~\ref{fig:sim:sc1}, \ref{fig:sim:sc2} and \ref{fig:sim:sc3} for the results.
\begin{enumerate}
    \item \emph{Landscapes for the settings in \citep{boumal2015low}.}
    We first illustrate the landscape of $\fp$ for the choice of sampling probabilities in \citep[\S5]{boumal2015low}: $p=0.026$ for Setting 1,  $p=0.006$ for Setting 2, and  $p=0.01$ for Setting 3. We compare them with the $p=1$ case. As one can see from the top-left plot of Figures~\ref{fig:sim:sc1}, \ref{fig:sim:sc2} and \ref{fig:sim:sc3}, the landscapes are very similar to  the fully observed case.
    
    \item \emph{Evolution of cost for varying $p$.} 
    We next plot the evolution of the expected cost value $\E\fp(\VX)$  as we vary the sampling probability $p$. For each setting, we generate  six  points $\{\VX_1,\dots, \VX_6\}$ randomly from $\Gr(m,r)$ whose principal angles to the ground truth matrix all lie between $\tfrac{\pi}{4}$ and $\tfrac{3\pi}{4}$.
   Then for each point $\VX_i$, we  compute $\E\fp(\VX_i)$ as we  vary the sampling probability $p$.
   To compute the expected cost value, we run  $100$ independent trials and  compute the average. 
   We also depict the error bars with shades.
   As one can see from  the top-right plot of Figures~\ref{fig:sim:sc1}, \ref{fig:sim:sc2} and \ref{fig:sim:sc3},  the cost values decrease while keeping the relative ratio similar as $p$ decreases. Also, notice that the cost values are well concentrated around their averages (the error bar shades are only noticeable for Setting 3).

   \item \emph{Comparison of landscape for varying $p$.} 
 For each setting, we compare the landscape of $\fp$ for sampling probabilities $p=1,~0.1,~0.01,~0.001$.
 As one can see from  the bottom plot of Figures~\ref{fig:sim:sc1}, \ref{fig:sim:sc2} and \ref{fig:sim:sc3},  the landscapes for $p=0.1,0.01$ are almost identical to those for $p=1$.
 Experiment 3 demonstrates that the landscapes of the partially observed case are greatly similar to those of the fully observed case unless $p$ is too small.
\end{enumerate}

\subsection{Other formulations}

In this section, we discuss other formulations considered in previous works~\citep{dai2012geometric,boumal2015low}.
First, we note that the cost function \eqref{cost:formulation} could become discontinuous in general.
To formally see this, let us revisit  \citep[Example 1]{dai2012geometric}.

\begin{example} Consider a toy example where the ground truth matrix $\VM$ is equal to $[0, \ 1, \ 1 ]^\top$ and $\pp = \{2,3\}$, i.e., the two entries of value $1$ are observed.
Now let us consider the cost at the point $\vx = [\sqrt{1-2\eps^2}, \ \eps, \ \eps ]^\top$ for $\eps\in [0,1/\sqrt{2}]$.
From \eqref{cost:formulation},  
\begin{align*}
    \fp(\vu)  &= \frac{3}{4}\cdot  \min_{y\in \R } \pnorm{y \vx   -[0, \ 1, \ 1 ]^\top}^2 = \frac{3}{2}\cdot\min_{y\in \R } (1-\eps y)^2\\ 
    &  =\begin{cases} 0 &\text{if  }\eps\in (0,1/\sqrt{2}],\\
    3/2 &\text{if }\eps=0.\end{cases}
\end{align*}
Hence, one can see that $\fp$ for this toy example is discontinuous at $\vx =[1, \ 0, \ 0 ]^\top$. \end{example}
\noindent To remedy this discontinuity issue, \citep{dai2012geometric} consider a different formulation based on the chordal distance, and \citep{boumal2015low} add a regularization term to the cost function.
These modifications are  indeed shown to be useful for real applications (see  \citep[\S6]{boumal2015low} or \citep[\S 3.4]{keshavan2009gradient}). 
    
\section{Discussion}
\label{sec:discussion}

We conclude this paper with several relevant open questions. 

\paragraph{Extension to the partially observed case.} 
We have seen in Section~\ref{sec:partial} that the landscapes for the partially observed case are greatly similar to those for the fully observed case unless $p$ is too small.
Hence, one future direction is to extend our theoretical results  in Section~\ref{sec:landscape} to the partially observed case.

\paragraph{Optimization aspect.} 
   Our main results  characterize  the Grassmannian landscape of \eqref{cost:formulation_full}. 
   It is then natural to ask the optimization aspect of our findings.
   Since our landscape analysis reveals the (strong) geodesic convexity of the cost within the basin $\NN{[\VU]}{\pi/4}$,    gradient methods converges fast within the basin.
   It would be then interesting to see how the gradient methods (or other optimization methods) behave outside the basin. 
  Our preliminary experiment on the trajectories of Riemannian gradient descent is reported in Figure~\ref{fig:optimization}. 
   Also, investigating other conditions for fast optimization  (e.g.,
   Polyak-{\L}ojasiewicz (P{\L}) inequality) would be an interesting direction.
   Lastly, whether one could develop a new optimization method based on  our results is also an intriguing direction to pursue.

\begin{figure}
    \centering 
    \includegraphics[width=.4\paperwidth]{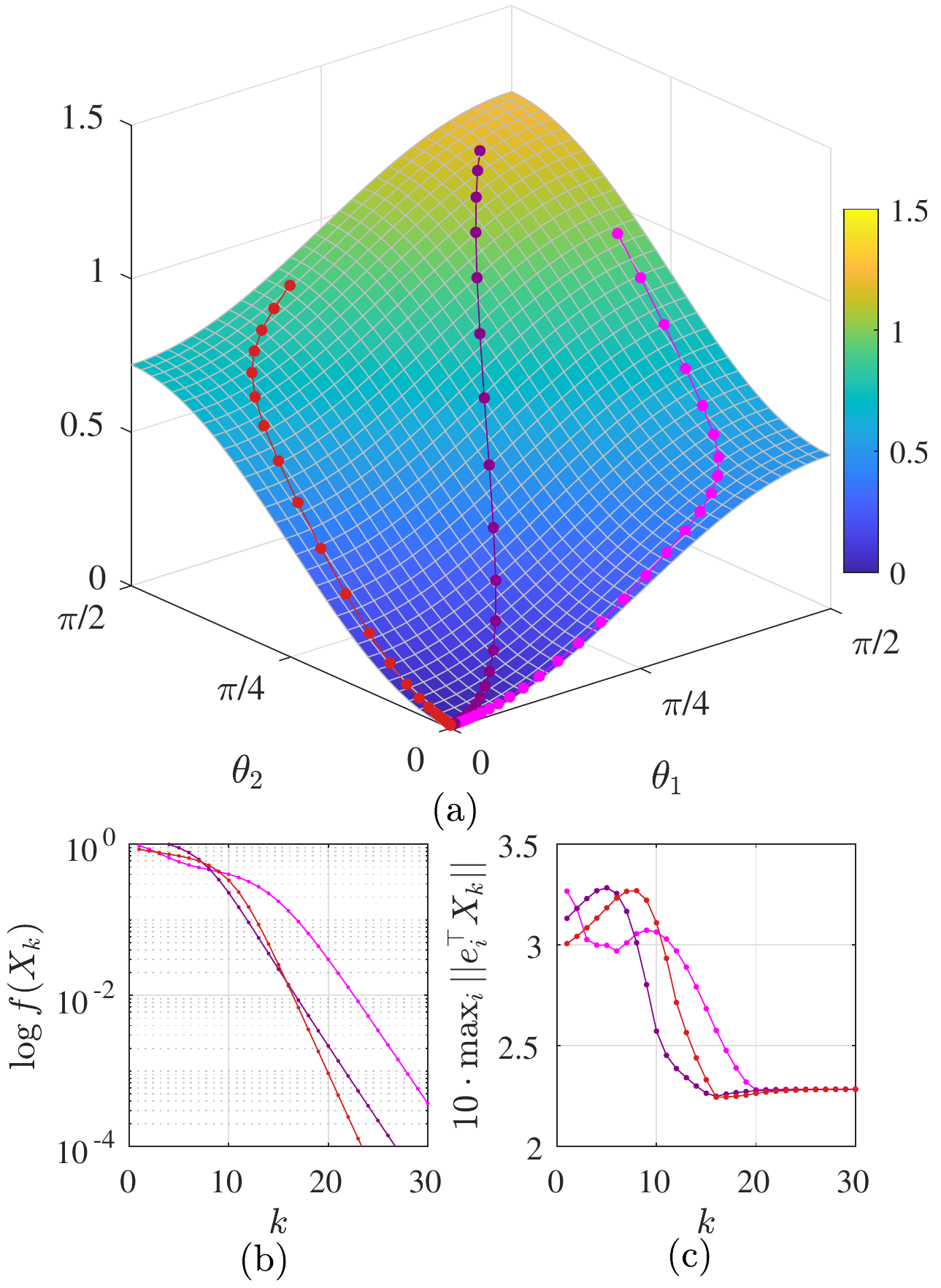}
    \caption{(a) Illustration of trajectories  of Riemannian gradient descents on  $\ff$.
    We choose the ground truth matrix to be a $100\times 200$ matrix of  rank $2$. The step-size is fixed $\eta=0.1$. (b) Objective values along the trajectories in log-scale. (c) The incoherence parameters along the trajectories.}
    \label{fig:optimization}
\end{figure}

\paragraph{Geometry of related problems.}   It is natural to ask whether one could characterize similar results for other related problems.
  In fact, the main observation in Ge et al.~\citep{ge2017no} is that their analysis for matrix completion also applies to other low rank recovery problems like matrix sensing and robust PCA.
  Based on this, we suspect that  similar landscape results could be characterized for these two problems.
  Another possibility is to study  robust subspace recovery~\citep{maunu2019well} in light of Remark~\ref{rmk:subspace}.

\paragraph{Formulations based on other distances.} 
    In Section~\ref{sec:projection}, we interpreted \eqref{cost:formulation_full} as a weighted version of the projection distance.
    Given this interpretation, it would be interesting to see whether matrix completion can be formulated using  other notions of distances between subspaces~\citep[\S4.3]{edelman1998geometry}.
    Would other distances lead to a better formulation? We note that this question is already partially explored by  \cite{dai2012geometric} where they consider the chordal distance for the purpose of \emph{consistent} matrix completion.
    
\section*{Acknowledgement}

We thank Suvrit Sra for stimulating discussions, especially regarding the geodesic convexity results (Lemma~\ref{lem:geodesic_convexity} and Corollary~\ref{cor:strongconvex}). We thank Tyler Maunu for pointing us to several related works. We thank Sinho Chewi for discussion regarding Remark~\ref{rmk:curvature} and also for various comments and suggestions on the manuscript. 
We thank Chen Lu for critical comments regarding the formulation.

\bibliographystyle{apalike} 
\bibliography{ref}

 \end{document}

%% file: 00notation.tex
\usepackage{graphicx} 
\usepackage{subfigure} 
\usepackage{xcolor} 
\usepackage{authblk}
\usepackage{mathtools} 

\usepackage{amsmath,amsbsy,amssymb,amsthm}
\usepackage{graphicx}
\usepackage{cases}
\usepackage{tikz}
\usepackage{subfigure} 
 \usepackage{hyperref}
\usepackage{color,colortbl}
\definecolor{darkblue}{rgb}{0.0,0.0,0.65}
\definecolor{darkred}{rgb}{0.65,0.0,0.0}
\hypersetup{
  colorlinks = true,
  citecolor  = darkblue,
  linkcolor  = darkred,
  filecolor  = darkblue,
  urlcolor   = darkblue,
}

\usepackage[noabbrev,nameinlink]{cleveref}

 \definecolor{emph}{gray}{0.95}

 \usetikzlibrary{shadows}

\newcommand\bref[3][darkred]{%
    \begingroup%
    \hypersetup{linkcolor=#1}%
    \hyperlink{#2}{#3}%
    \endgroup}

\theoremstyle{plain}
\newtheorem{theorem}{Theorem}

\newtheorem{corollary}[theorem]{Corollary}

\newtheorem{proposition}{Proposition}
\newtheorem{lemma}[proposition]{Lemma}
\newtheorem*{claim}{Claim}
\newtheorem{remark}{Remark} 

\theoremstyle{definition}
\newtheorem{example}{Example}
\newtheorem{definition}{Definition}

\newtheoremstyle{named}%
    {}{}{\itshape}{}{\bfseries}{.}{.5em}{\thmnote{#3}}
\theoremstyle{named}
\newtheorem*{custom3}{Convenient setup}

\numberwithin{equation}{section}

\newcommand{\setup}{{\bref{setup}{{\sf \small principal alignment}}}}

\newcommand{\st}{_{\sf p}}

\newcommand{\QQ}{\mathbf{\Xi}}
\newcommand{\qq}{\xi}
\newcommand{\eps}{\epsilon}

\newcommand{\argmin}{\mathop{\rm argmin}}

\newcommand{\mini}{\mathop{\rm minimize}}

\newcommand{\NN}[2]{\mathcal{N}_{#1}(#2)}

\def\Tr{{\sf Tr}}

\newcommand{\R}{\mathbb{R}}  
\newcommand{\E}{\mathop{\mathbb{E}}}

\newcommand{\Gr}{\mathsf{Gr}} 
 
\newcommand{\col}{\mathsf{col}} 
\newcommand{\diag}{\mathsf{diag}}

\newcommand{\dc}{\mathbf{\Delta}}

\newcommand{\dcc}{\mathbf{\Delta_{{\sf crit}}}}

\newcommand{\dcn}{\bar{\mathbf{\Delta}} }

\newcommand{\dd}{\mathbf{\Delta}\st}

\newcommand{\darc}{\mathsf{d_{arc}}}
\newcommand{\dchor}{\mathsf{d_{chor}}}
\newcommand{\dproj}{\mathsf{d_{proj}}}

\newcommand{\norm}[1]{\left\|{#1}\right\|}
\newcommand{\fnorm}[1]{\left\|{#1}\right\|_{\sf{F}}}
\newcommand{\pnorm}[1]{\left\|{#1}\right\|_{\Omega}}

\newcommand{\nn}{\nonumber}

\newcommand{\vxxx} {\mathbf{x}\st}
\newcommand{\vyyy} {\mathbf{y}\st}

\newcommand{\vze}{\mathbf{0}}
\newcommand{\vx}{\mathbf{x}}
\newcommand{\vy}{\mathbf{y}}

\newcommand{\vu}{\mathbf{u}}

\newcommand{\vv}{\mathbf{v}}

\newcommand{\vg}{\mathbf{g}}

\newcommand{\vla}{\boldsymbol{\lambda}}
\newcommand{\sig}{\sigma}
\newcommand{\smax}{\sigma_{{\sf max}}}
\newcommand{\smin}{\sigma_{{\sf min}}}

\newcommand{\VX}{\mathbf{X}}
\newcommand{\VY}{\mathbf{Y}}
\newcommand{\VYX}{\mathbf{Y}_{\mathbf{X}}}

\newcommand{\VU}{\mathbf{U}}
\newcommand{\VS}{\mathbf{\Sigma}}
\newcommand{\VV}{\mathbf{V}}

\newcommand{\VM}{\mathbf{M}} 
 
\newcommand{\VD}{\mathbf{D}} 
\newcommand{\VG}{\mathbf{G}} 
\newcommand{\VI}{\mathbf{I}} 
\newcommand{\VR}{\mathbf{R}} 
\newcommand{\VP}{\mathbf{P}}
\newcommand{\VQ}{\mathbf{Q}} 
\newcommand{\VA}{\mathbf{A}} 
\newcommand{\VB}{\mathbf{B}} 
 
\newcommand{\VZE}{\mathbf{O}} 
\newcommand{\VTH}{\mathbf{\Theta}}
\newcommand{\VPH}{\mathbf{\Phi}}
\newcommand{\VLA}{\mathbf{\Lambda}}

\newcommand{\pp}{\Omega}

\newcommand{\orth}{\mathsf{O}}
\newcommand{\T}[1]{{\sf T}_{#1}} 
\newcommand{\ff}{f_{\sf full}} 
\newcommand{\fp}{f }

\newcommand{\gf}{g_{\sf full}} 
\newcommand{\gp}{g}

\def\grad{\mathsf{grad}}
\def\hess{\mathsf{hess}}

\newcommand{\inp}[2]{\left\langle#1, #2 \right\rangle}
\newcommand{\inpp}[2]{\left\langle#1, #2 \right\rangle_{\Omega}}

%% file: 11abstract.tex
We study the non-convex matrix factorization approach to matrix completion via Riemannian geometry. Based on an optimization formulation over a Grassmannian manifold, we characterize the landscape based on the notion of principal angles between subspaces. For the fully observed case, our results show that there is a region in which the cost is geodesically convex, and outside of which all critical points are strictly saddle. We empirically study the partially observed case based on our findings.